\documentclass[oneside,11pt]{article}
\usepackage{graphicx, subfigure}
\usepackage[top=1in,bottom=1in,left=1in,right=1in]{geometry}
\usepackage[sort&compress,numbers]{natbib} 
\usepackage{amsfonts, amsmath, amssymb, amsthm, constants, bbm}
\usepackage{enumitem}

\usepackage[mathscr]{euscript}
\usepackage{pifont}
\usepackage{mathtools}

\usepackage{comment}
\usepackage{float}
\usepackage{tikz}

\usepackage{color}
\definecolor{darkred}{RGB}{100,0,0}
\definecolor{darkgreen}{RGB}{0,100,0}
\definecolor{darkblue}{RGB}{0,0,150}

\usepackage{hyperref}
\usepackage{url}

\newtheorem{thm}{Theorem}

\newtheorem{lem}{Lemma}

\theoremstyle{remark}

\newtheorem{rem}{Remark}

\def\beq{\begin{equation}} 
\def\eeq{\end{equation}}
\def\beqn{\begin{eqnarray*}}
\def\eeqn{\end{eqnarray*}}
\def\Bitem{\begin{itemize}\setlength{\itemsep}{.2in}}
\def\bitem{\begin{itemize}\setlength{\itemsep}{.05in}}
\def\eitem{\end{itemize}}
\def\Benum{\begin{enumerate}\setlength{\itemsep}{.2in}}
\def\benum{\begin{enumerate}\setlength{\itemsep}{.05in}}
\def\eenum{\end{enumerate}}
\def\bmult{\begin{multline*}}
\def\emult{\end{multline*}}
\def\bcenter{\begin{center}}
\def\ecenter{\end{center}}
\def\bframe{\begin{frame}}
\def\eframe{\end{frame}}

\newcommand{\thmref}[1]{Theorem~\ref{thm:#1}}

\newcommand{\secref}[1]{Section~\ref{sec:#1}}

\def\cJ{\mathcal{J}}

\def\cL{\mathcal{L}}

\def\cQ{\mathcal{Q}}

\def\bbI{\mathbb{I}}
\def\bbJ{\mathbb{J}}

\def\bbN{\mathbb{N}}

\def\bbR{\mathbb{R}}

\def\bbT{\mathbb{T}}

\def\bbZ{\mathbb{Z}}

\newcommand{\E}{\operatorname{\mathbb{E}}}
\renewcommand{\P}{\operatorname{\mathbb{P}}}

\def\eps{\varepsilon}

\def\comp{\mathsf{c}}
\def\1{\mathbbm{1}}

\def\uj{U_{(j)}}
\def\ui{U_{(i)}}

\def\buj{\bar{U}_{(j)}}
\def\bui{\bar{U}_{(i)}}

\definecolor{purple}{rgb}{0.4,.1,.9}

\newcommand\blfootnote[1]{%
  \begingroup
  \renewcommand\thefootnote{}\footnote{#1}%
  \addtocounter{footnote}{-1}%
  \endgroup
}

\begin{document}
\thispagestyle{empty}

\title{\sc{On the Asymptotic Distribution of the Scan Statistic for Empirical Distributions}}
\author{Andrew Ying \and Wen-Xin Zhou}
\date{}
\maketitle
\blfootnote{All authors are with the Department of Mathematics, University of California, San Diego, USA.  Contact information is available \href{http://www.math.ucsd.edu/\~anying}{here} and \href{http://www.math.ucsd.edu/\~wez243}{here}.}
\begin{abstract}
We investigate the asymptotic behavior of several variants of the scan statistic applied to empirical distributions, which can be applied to detect the presence of an anomalous interval with any length. 
Of particular interest is Studentized scan statistic that is preferable in practice.
The main ingredients in the proof are Kolmogorov's theorem, a Poisson approximation, and recent technical results by \citet{kabluchko2014limiting}.
\end{abstract}








\section{Introduction}
The study of the scan statistics dates back\footnote{ Naus himself cites even earlier work in the 1940's by \citet{silberstein1945xliii}, \citet{berg1945xliv}, and \citet{mack1948xc}.} to \citet{naus1965distribution}, who derived the probability that an interval of a certain length contains a certain fraction of independent and identically distributed (iid) samples from the uniform distribution on $[0,1]$. Specifically, let $U_1, \dots, U_n$ be iid random variables from Unif$(0,1)$ with empirical cumulative distribution function (CDF) denoted by $F_n$, and let $h$ be the length of the underlying interval of interest. \citet{naus1965distribution} studied the distribution of
\beq\label{scan0}
\sup_{0 \le a \le 1} F_n(a+h) - F_n(a).
\eeq
The scan statistics of the uniform empirical distributions can be used to detect elevated signal relative to any continuous null distribution, after an appropriate inverse CDF transformation. Knowing this distribution \eqref{scan0} is essential to calibrating the scan statistic in the context of detecting, in a uniform background, the presence of an interval of a certain length with an unusually high density of points.  This is considered today a quintessential detection problem, with applications in the detection of disease clusters \cite{besag1991detection} and syndromic surveillance \cite{heffernan2004ssp}, among many others \cite{glaz2001scan,glaz2009scan,glaz2012scan, handbook}.

In practice, even in the simplest case where only a single anomalous interval may be present, the length of that interval is almost always unknown.  In that case, it is natural to consider intervals of various lengths, but standardize the counts, leading to 
\beq\label{scan}
\sup_{0 \le a \le 1} \sup_{h_- \le h \le h_+} \frac{\sqrt{n}(F_n(a+h) - F_n(a) - h)}{\sqrt{h (1-h)}}.
\eeq
This can be seen to approximate the likelihood ratio test \cite{kulldorff1997spatial}.  The parameters $h_-$ and $h_+$ limit the search to intervals that are neither too short and nor too large.  The main goal of this paper is to derive the asymptotic (as $n \to \infty$) distribution  of \eqref{scan} along with its studentized counterpart
\beq\label{studentizedscan}
\sup_{0 \le a \le 1} \sup_{h_- \le F_n(a+h) - F_n(a) \le h_+} \frac{\sqrt{n}(F_n(a+h) - F_n(a) - h)}{\sqrt{(F_n(a+h) - F_n(a) ) (1- F_n(a+h) + F_n(a) )}}.
\eeq

\begin{rem}
From the four theorems in Section \ref{sec:main}, one finds out that relatively small scale $h$ dominates in \eqref{scan} and \eqref{studentizedscan}. There are works that apply scale corrections to the scan \cite{dumbgen2001multiscale, sharpnack2016exact, konig2018multidimensional}. Under the scale corrections, one scale no longer dominates.
\end{rem}

\subsection{Related work: point processes}
In one of the most celebrated results in what is now the empirical process literature, \citet{kolmogorov1933sulla} derived the limiting distribution of $\sqrt{n} \, \sup_{0 \le a \le 1} (F_n(a) - a)$.  This is the Kolmogorov-Smirnov statistic, and it can be seen as scanning over intervals of the form $[0, a]$, $0 \le a \le 1$.  

For similar reasons that motivated the introduction of the normalized scan statistic \eqref{scan} as an improvement over the unnormalized one \eqref{scan0}, \citet{anderson1952asymptotic} introduced and studied normalized variants of the Kolmogorov-Smirnov statistic, some of them of the form $\sqrt{n} \sup_a (F_n(a) - a) \sqrt{\psi(a)}$, where $\psi$ is a given weight function.  The choice $\psi(a) = [a (1-a)]^{-1}$ is particularly compelling, leading to the statistic 
\beq\label{AD}
\sup_{0 \le a \le 1} \frac{\sqrt{n} (F_n(a) - a)}{\sqrt{a(1-a)}}.
\eeq
\citet{eicker1979asymptotic} and \citet{jaeschke1979asymptotic}  obtained the limiting distributions of this statistic, its variants of the form
\beq\label{vn}
V_n = \sup_{\epsilon_n \le a \le \delta_n}\frac{\sqrt{n} (F_n(a) - a)}{\sqrt{a(1 - a)}},
\eeq
and its Studentized counterpart
\beq\label{hvn}
\hat V_n = \sup_{\epsilon_n \le a \le \delta_n}\frac{\sqrt{n} (F_n(a) - a)}{\sqrt{F_n(a)(1 - F_n(a))}},
\eeq
for some given $0 \le \epsilon_n \le \delta_n \le 1$. 
We note that these statistics can be directly expressed in terms of the order statistics, $U_{(1)} \le \cdots \le U_{(n)}$, which when $\eps_n = 0$ and $\delta_n = 1$, is as follows
\beq\label{hc}
\max_{1 \le i \le n}\frac{i - n\ui}{\sqrt{n\ui(1 - \ui)}},
\eeq
and
\beq\label{hcstu}
\max_{1 \le i < n}\frac{i - n\ui}{\sqrt{i(1 - \frac{i}{n})}},
\eeq
respectively.

\citet{berk1979goodness} proposed to directly look at each order statistic individually, combining the resulting tests using Tippett's method, leading to 
\beq
\min_{1 \le i \le n} B(\ui; i, n - i + 1),
\eeq
with $B(\cdot; a, b)$ denoting the distribution function of the Beta$(a, b)$ distribution.  \citet{moscovich2016exact} and \citet{gontscharuk2017asymptotics} derived the asymptotic distribution of this statistic. Other goodness-of-fit tests include the reversed Berk-Jones statistic \cite{jager2004new} and Phi-divergence tests \cite{jager2007goodness}, etc.

We note that the two-sided version of the above-mentioned tests have been considered and studied.

\subsection{Related work: signals}
Closely related to the work above is the setting where, instead of observing a point cloud, one observes a signal.  The simplest situation is that of a one-dimensional signal defined on a regular lattice, that is, of the form $X_1, \dots, X_n$.  The null situation is when these are iid from some underlying distribution on the real line, for example, the standard normal distribution. By writing
\beq
R_n = \max_{1 \le i \le n - k} \{S_{i + k} - S_i\},~~~k = \lfloor c \log n\rfloor,~~ c > 0,
\eeq
\citet{erdos1970new} investigated the strong limit of $R_n/(\alpha k)$ when $X_1$ has a finite moment generating function around a neighborhood of zero. \citet{deheuvels1986exact} studied $\limsup$ and $\liminf$ of $(R_n - \alpha k)/\log k$. When the goal is to detect an interval where the observations are unusually large, and the length of the (discrete) interval is unknown, it becomes of interest to study the following scan statistic
\beq\label{maxincre}
Z_n = \max_{1 \le i < j \le n} \frac{S_j - S_i}{\sqrt{j - i}},
\eeq
where $S_k = \sum_{i = 1}^k X_i$.

The study of such statistics dates back to the work of \citet{darling1956limit}, who derived the limiting distribution of
\beq
\max_{1 \le j \le n} \frac{S_j}{\sqrt{j}},
\eeq
which can be seen as scanning intervals of the form $\{1, \dots, j\}$.

\citet{siegmund1995using} provided the limiting distribution of the statistic \eqref{maxincre} under the assumption that the $X_i$'s are iid normal.  This study was extended by \citet{mikosch2010limit} to the case where the underlying distribution is heavy-tailed, and by \citet{kabluchko2014limiting} when the underlying distribution has finite moment generating function in a neighborhood of the origin.
\citet{kabluchko2011extremes} generalized the result to the multivariate setting where the variables are indexed by a multi-dimensional lattice; see also \cite{sharpnack2016exact, konig2018multidimensional}.
\citet{proksch2018multiscale} studied more general scanning procedures motivated within the framework of inverse problems.

There is a parallel literature for continuous processes, where one observes instead $X_t, t \in [0,1]$ (in dimension 1).  See, for example, \citet{aldous2013probability, qualls1973asymptotic} and \citet{chan2006maxima}.

\subsection{Related work: Lipschitz-1/2 modulus of the uniform empirical process}
The results of \citet{mason1983strong} on the Lipschitz-1/2 modulus of the uniform empirical process, defined by
\beq\label{eq:osci}
\sup_{0 \le a \le 1 - h} \sup_{t \le h \le 1} \frac{\sqrt{n}|F_n(a+h) - F_n(a) - h|}{\sqrt{h}},
\eeq
are most closely related to the present results. They proved strong limit theorems for \eqref{eq:osci} with $t = t_n \to 0$ at various rates. We refer to \citet[Chapter 14.2]{shorack2009empirical} for a review.

\subsection{Content}

The rest of the paper is organized as follows. 
We state our main results in \secref{main}, where we provide asymptotic distributions of some scan statistics and their variants. The proofs are provided in \secref{proof}. 

\section{Main results} \label{sec:main}
Recall that $U_1,\ldots,U_n$ are iid from the uniform distribution on $[0, 1]$, and that $U_{(1)} \leq \cdots \leq U_{(n)}$ denote the order statistics. (Whenever needed, we write $U_{(0)} \equiv 0$ and $U_{(n + 1)} \equiv 1$.)  

\subsection{Studentized scan statistics}
We derive the asymptotics for \eqref{studentizedscan} before \eqref{scan} for convenience of the proof. As we did earlier, we may rewrite \eqref{studentizedscan} directly in terms of the order statistics, in the form of 
\beq\label{mnpluskl}
M_n^+(k, l)  = \max_{0 \le i < j \le n:\, k \le j - i < l} M_{i, j},
\eeq
where
\beq\label{normalizedorder}
M_{i, j}  = \frac{j - i - n(\uj - \ui)}{\sqrt{(j - i)(1 - \frac{j - i}{n})}}.
\eeq
We will be particularly interested in the following special case
\beq\label{stats:mnplus}
M_n^+ := M_n^+(1, n),
\eeq
which is the analog of \eqref{hcstu}.
Not surprisingly, the limiting distribution is an extreme value distribution, specifically, a Gumbel distribution. 
Indeed, we have the following. 

\begin{thm}\label{thm:mnplus}
For any $\tau \in \bbR$, 
\beq\label{result:mnplus}
\lim_{n \to \infty}\P \bigg\{ M_n^+ \le \sqrt{2 \log n} - \frac{3\log\log n}{2\sqrt{2 \log n}} + \frac{\tau}{\sqrt{2 \log n}} \bigg\}  = \exp\big(-c\, \exp(-\tau)\big),
\eeq
where $c = \tfrac{8}{9\sqrt{\pi}}$.
\end{thm}

Similarly, define the opposite one-sided statistics
\beq
M_n^-(k, l)  = -\min_{0 \le i < j \le n:\, k \le j - i \le l} M_{i, j},
\eeq
and
\beq\label{stats:mnminus}
M_n^- := M_n^-(1, n).
\eeq
Finally, define the two-sided statistics
\beq
M_n(k, l) = \max\{M_n^+(k, l), M_n^-(k, l)\} = \max_{0 \le i < j \le n:\, k \le j - i < l} |M_{i, j}|,
\eeq
and 
\beq\label{stats:mn}
M_n := M_n(1,n) = \max\{M_n^+, M_n^-\}. 
\eeq

For these statistics too, the limiting distribution is a Gumbel distribution, but what is surprising here is that these statistics do not behave the same way as $M_n^+$.  In particular, $M_n^- = (1 + o_P(1)) \log n$, and therefore dominates $M_n^+$ in the large-sample limit, implying that $M_n = M_n^-$ with probability tending to~1.
Indeed, we have the following.

\begin{thm} \label{thm:mnminus}
For any $\tau \in \bbR$, 
\beq \label{result:mnminus}
\lim_{n \to \infty}\P\big\{ M_n^- \le \log n + \tau \big\} =  \exp(-\exp(1 - \tau)).
\eeq
Moreover, 
\beq \label{result:mn}
\lim_{n \to \infty} \P\big\{ M_n = M_n^- \big\} = 1.
\eeq
\end{thm}

\subsection{Standardized scan statistics}
We also examine the large-sample behavior of standardized scan statistics \eqref{scan}. Following the same way as rewriting \eqref{studentizedscan} before. Define
\beq
\tilde M_n^+(k, l) := \max_{0 \le i < j \le n: k \le j - i \le l}\tilde M_{i, j},
\eeq
where
\beq
\tilde M_{i, j} := \frac{j - i - n(\uj - \ui)}{\sqrt{n(\uj - \ui)(1 - \uj + \ui)}}.
\eeq
Note that 
\beq\label{stats:tmnplus}
\tilde M_n^+ := \tilde M_n^+(1, n),
\eeq
is the analog of \eqref{hc}.

The behavior of $\tilde M_n^+$ turns out to be very different from that of its studentized analog $M_n^+$.  However, we recover a similar behavior if we appropriately bound the length of the scanning interval from below.

\begin{thm}\label{thm:tmnplus}
For any $\tau > 0$,
\beq\label{eq:tmnplusall}
\lim_{n \to \infty}\P \bigg\{ \tilde M_n^+ \le \sqrt{\frac{n}{\tau}} \bigg\} = \exp(-\tau).
\eeq
Moreover, for any $A > 0$, defining $k_n = \lceil A(\log n)^3 \rceil$,
\beq\label{eq:tmnpluscubelocal}
\lim_{n \to \infty}\P\bigg\{  \tilde M_n^+(k_n, n) \le \sqrt{2 \log n} - \frac{3\log\log n}{2\sqrt{2 \log n}} + \frac{\tau}{\sqrt{2 \log n}} \bigg\}  = \exp(- c_A\, \exp(-\tau)) ,
\eeq
where $c_A = \int_A^\infty \Lambda_1(a)da$ with $\Lambda_1(a) = \frac{1}{2\sqrt{\pi}a^2} \exp\big(\frac{\sqrt{2}}{3\sqrt{a}}\big)$.
\end{thm}


\begin{rem}
Here we choose $k_n \propto (\log n)^3$ because we want to examine the behavior of $\tilde M^+(K, L)$, compared to its counterpart $M^+(K, L)$ at the most contributed part, which is reflected in the proof of Theorem \ref{thm:mnplus}. For readers who are curious about other choices of $k_n$, we note that $\tilde M_{i, j}$ behaves like subgaussian, or named as ``sublogarithmic'' in \cite{kabluchko2014limiting}. Roughly speaking, $\tilde M_n^+(k_n, n)$ will likely to take its maximum around the indices $i$, $j$ with small length, that is, when $j - i$ is close to $k_n$.
\end{rem}

Define the standardized analog of \eqref{stats:mnminus}
\beq
\tilde M_n^-(k, l) = -\min_{0 \le i < j \le n:\, k \le j - i \le l}\tilde M_{i, j},
\eeq
with
\beq\label{stats:tmnminus}
\tilde M_n^- := \tilde M_n^-(1, n),
\eeq
as well as the analog of \eqref{stats:mn}
\beq
\tilde M_n(k, l) = \max\{\tilde M_n^+(k, l), \tilde M_n^-(k, l)\},
\eeq
with
\beq\label{stats:tmn}
\tilde M_n := \tilde M_n(1, n) = \max\{\tilde M_n^+, \tilde M_n^-\}.
\eeq

\begin{thm}\label{thm:tmnminus}
We have
\beq
\lim_{n \to \infty} \P\big( \tilde M_n = \tilde M_n^+ \big) = 1.
\eeq
Thus for any $\tau > 0$, 
\beq\label{eq:tmnsall}
\lim_{n \to \infty}\P \bigg( \tilde M_n \le \sqrt{\frac{n}{\tau}} \bigg) = \exp(-\tau).
\eeq

\end{thm}

\begin{rem}
While the behavior of the Studentized statistic $M_n^+$ is driven by the smallest intervals, this is not as much the case for the standardized statistic $\tilde M_n^+$.  Indeed, a large value of $M_n^+$ comes from some $n(\uj - \ui)$ being large compared to $j - i$, however, $n(\uj - \ui)$ being in the denominator defining $\tilde M_n^+$, its impact is lessened.
\end{rem}


\section{Proofs of Main Results}\label{sec:proof}

Our proof arguments are based on standard moderate and large deviation results, Kolmogorov's theorem, a Poisson approximation \cite{arratia1989two}, as well as some technical results developed by  \citet{kabluchko2014limiting} in their study of the limiting distribution of the scan statistic in the form of \eqref{maxincre}.

\subsection{Preliminaries}
Throughout the paper, we assume that $\{X_k, k \in \bbZ\}$ are iid distributed with the density, 
\beq
f(x) = \mathbbm{1}(x \leq 1)\exp(x-1),
\eeq
noting that $-X_1 + 1$ follows standard exponential distribution.  This distribution has zero mean and unit variance. 
Define the two-sided partial sums, 
\beq
S_k^+ = \sum_{i=1}^k X_i, ~~~~ S_0^+= 0, ~~~~S_{-k}^+ = -\sum_{i=1}^k X_{-i}, ~~k \in \bbN
\eeq
and
\beq
S_k^- := - S_k^+. 
\eeq
They will play a central role in what follows. Define the normalized increments
\beq
Z_{i, j}^\pm = \frac{S_j^\pm - S_i^\pm}{\sqrt{j - i}},
\eeq
\beq
Z_n^\pm(k, l) := \max_{1 \le i < j \le n: k \le j - i \le l} Z^\pm_{i, j}, \qquad
Z_n^\pm := Z_n^\pm(1, n).
\eeq
Let $\varphi^\pm(t)$ be the cumulant generating functions of $\pm X_1$ respectively.  We have
\beq
\varphi^+(t) = t - \log(1 + t), \quad \text{if } t \ge 0.
\eeq
\beq
\varphi^-(t) = \begin{cases}
-t - \log(1 - t), \quad &\text{if }0 \le t \le 1,\\
\infty, &\text{if }t \ge 1,
\end{cases}
\eeq
Also, define $I^+(s)$ and $I^-(s)$ as the respective Legendre-Fenchel transforms (a.k.a., rate functions).  We have
\beq
I^+(s) = \begin{cases}
-s - \log(1 - s), &\text{if } 0 \le s \le 1,\\
\infty, &\text{if } s \ge 1,
\end{cases}
\eeq
and
\beq
I^-(s) = s - \log(1 + s),
\eeq
with respective Taylor expansions at $0$ (as $s \to 0$)
\begin{align*}
I^+(s) &= s^2/2 + s^3/3+ o(s^3),\\
I^-(s) &= s^2/2 - s^3/3+ o(s^3).
\end{align*}

We also prepare several usefull lemmas. The first two lemmas are well-known moderate and large deviations results \cite{cramer1938sommes, bahadur1960deviations}.  

\begin{lem}\label{lem:moderatedev}
Let $(x_k)$ be a sequence satisfying $x_k \to \infty$ and $x_k = o(\sqrt{k})$ as $k \to \infty$.  Then, as $k \to \infty$,
\beq
\P\bigg(  \frac{S_k^\pm}{\sqrt{k}}  \ge x_k \bigg)  \sim \frac{1}{\sqrt{2\pi}x_k} \exp\bigg\{-k I^\pm \bigg(\frac{x_k}{\sqrt{k}}  \bigg) \bigg\}.
\eeq
\end{lem}
\begin{lem}
For every $k \in \bbN$ and $x > 0$, we have
\beq\label{eq:rescumulantineq}
 \P\bigg(  \frac{S_k^\pm}{\sqrt{k}} \ge x  \bigg)  \le \exp\bigg\{-k I^\pm \bigg( \frac{x}{\sqrt{k}}  \bigg) \bigg\} .
\eeq
Moreover, for every $A \le s_\infty$, where $s_\infty = \sup\{s \in \bbR: \P(X_1 \le s) \le 1\}$, there is $C_A > 0$ such that, for all $k \in \bbN$ and $x \in (0, A\sqrt{k})$,
\beq\label{eq:cumulantineq}
 \P\bigg(  \frac{S_k^\pm}{\sqrt{k}} \ge x  \bigg)  \le \frac{C_A}{x}  \exp\bigg\{-k I^\pm \bigg( \frac{x}{\sqrt{k}}  \bigg) \bigg\},
\eeq
\end{lem}

The following result is obtained from a simple application of Theorem 2.4 in \cite{petrov1995limit}, which provides an upper bound of the tail distribution of $\max_{1 \le k \le n}S_k^\pm$ by that of $S_n^\pm$.
\begin{lem}\label{lem:petrov}
We have
\beq
\P\bigg\{\max_{1 \le k \le n}S_k^\pm \ge x\bigg\} \le 2 \P\Big\{ S_n^\pm \ge x - \sqrt{2(n - 1)} \Big\}.
\eeq
\end{lem}


For completeness, we include Lemma 4.4 and 4.5 from \citep{kabluchko2014limiting} below. For integers $r > 0$ and $x < y$, define
\beq\label{T}
\bbT_r(x, y) := \big\{(i, j) \in \bbI: x - r \le i \le x\text{ and }y \le j \le y + r\big\}.
\eeq
\begin{lem}\label{lem:weakcontrolzijplus}
Fix constants $B_1, B_2 > 0$. Then for all $x \in \bbZ$, $l, r \in \bbN$ and all $u > 0$  such that $B_1l > u^2$ and $r \le B_2lu^{-2}$, we have
\beq
\cQ(l, r, u) := \P\bigg(  \max_{i, j \in \bbT_r(x, x + l)}\frac{S_j^+ - S_i^+}{\sqrt{l}}\ge u\bigg) \le \frac{C}{u}\exp\bigg(-\frac{u^2}{2} - \frac{c u^3}{\sqrt{l}} \bigg),
\eeq
where the constants $c$ and $C$ depend on $B_1$ and $B_2$ but do not depend on $x$, $l$, $r$, $u$.
\end{lem}

\begin{lem}\label{lem:interchange}
Let $\nu$, $\nu_n$, $n \in \bbN$, be measures on $[0, \infty)$ which are finite on compact intervals. Let $G$, $G_n$, $n \in 
\bbN$, be measurable functions on $[0, \infty)$ which are uniformly bounded on compact intervals. Assume that
\begin{enumerate}
    \item $\nu_n$ converges to $\nu$ weakly on every interval $[0, t]$, $t \ge 0$;
    \item for $\nu$-a.e. $s \ge 0$, we have $\lim_{n \to \infty} G_n(s_n) = G(s)$, for every sequence $s_n \to s$; 
    \item $\lim_{T \to \infty} \int_T^\infty |G_n|d\nu_n = 0$ uniformly when $n \ge N$ for some $N \in \bbN$.
\end{enumerate}
Then, $\lim_{n \to \infty} \int_0^T G_n d\nu_n = \int_0^T G d\nu$.
\end{lem}

We also provide an upper bound of the tail distribution $\max_{i, j \in \bbT_r(x, x + l)}(S_j^- - S_i^-)/\sqrt{l}$ also, which is cruder than its counterpart for $S_k^+$ in Lemma \ref{lem:weakcontrolzijplus} but shall suffice for our purposes.
\begin{lem}\label{lem:weakcontrolzijminus}
For all $x \in \bbZ$, $l, r \in \bbN^+$ and all $u > 40$  such that $l > u^2r$ and $r > 10u^2$, we have
\beq
\cQ(l, r, u) := \P\bigg(  \max_{i, j \in \bbT_r(x, x + l)}\frac{S_j^- - S_i^-}{\sqrt{l}}\ge u\bigg) \le C\exp\bigg(-\frac{u^2}{3}\bigg),
\eeq
where the constant $C$ does not depend on $x$, $l$, $r$, $u$.
\end{lem}
\begin{proof}
Before we proceed into the proof, one fact about $I^-(s)$ is
\beq\label{eq:iminusbound}
I^-(s) \ge \frac{1.01s^2}{3}, ~~~0 \le s \le 0.5,
\eeq
which can be easily checked. Define $V_{l,u} :=  u^2 - uS_l^-/\sqrt{l}$, $S_{k_1}^{(1)-}$ and $S_{k_2}^{(2)-}$ to be two partial sums of $-X_i$ independent of each other and $S_l^-$. With translation invariance, we bound $\cQ(l, r, u)$ as follows,
\begin{align*}
\cQ(l, r, u) &= \P\bigg(  \max_{i, j \in \bbT_r(0, 0 + l)}\frac{S_j^- - S_i^-}{\sqrt{l}}\ge u\bigg)\\
&= \P\bigg(\max_{0 \le k_1, k_2 \le r} \frac{S_{k_1}^{(1)-} + S_{k_2}^{(2)-}}{\sqrt{l}} + \frac{S_l^-}{\sqrt{l}} \ge u\bigg)\\
&= \P\bigg(\max_{0 \le k_1, k_2 \le r} \frac{S_{k_1}^{(1)-} + S_{k_2}^{(2)-}}{\sqrt{l}}  \ge \frac{V_{l, u}}{u}\bigg)\\
&\le \P\bigg(\max_{0 \le k_1, k_2 \le r} \frac{S_{k_1}^{(1)-} + S_{k_2}^{(2)-}}{\sqrt{l}} \ge \frac{V_{l, u}}{u}, V_{l, u} \le u^2\sqrt{\frac{r}{l}}\bigg)\\
&+ \P\bigg(\max_{0 \le k_1, k_2 \le r} \frac{S_{k_1}^{(1)-} + S_{k_2}^{(2)-}}{\sqrt{l}} \ge \frac{V_{l, u}}{u}, V_{l, u} > u^2\sqrt{\frac{r}{l}}\bigg)\\
&\le \P(V_{l, u} \le u^2\sqrt{r/l})+ \P\bigg(\max_{0 \le k_1, k_2 \le r} \frac{S_{k_1}^{(1)-} + S_{k_2}^{(2)-}}{\sqrt{l}} > u\sqrt{\frac{r}{l}}\bigg),
\end{align*}
where we bound these two terms individually. By the assumptions on $u, l, r$, we have $u(1 - \sqrt{r/l})/\sqrt{l} \le 0.5$. Thus with \eqref{eq:rescumulantineq} and \eqref{eq:iminusbound}, we have
\beq
\P\bigg(V_{l, u} \le u^2\sqrt{\frac{r}{l}}\bigg) = \P\bigg(\frac{S_l^-}{\sqrt{l}} \ge u - u\sqrt{\frac{r}{l}}\bigg) \le \exp\bigg[ -lI^+\bigg\{\frac{u(1 - \sqrt{r/l})}{\sqrt{l}}\bigg\}\bigg] \le \exp\bigg(-\frac{u^2}{3}\bigg).
\eeq
Now we switch to the second item, with Lemma \ref{lem:petrov}, \eqref{eq:rescumulantineq} and assumption that $r > 10u^2$, $u > 40$,
\begin{align*}
\P\bigg(\max_{0 \le k_1, k_2 \le r} \frac{S_{k_1}^{(1)-} + S_{k_2}^{(2)-}}{\sqrt{l}} \ge u\sqrt{\frac{r}{l}}\bigg) &\le 2\P\bigg(\max_{0 \le k \le r} \frac{S_k^-}{\sqrt{r}} \ge \frac{u}{2}\bigg)\\
&\le 4\P\bigg(\frac{S_r^-}{\sqrt{r}} \ge \frac{u}{2} - \sqrt{2}\bigg)\\
&\le C\exp\bigg\{-rI^-\bigg(\frac{u - 2\sqrt{2}}{2\sqrt{r}}\bigg)\bigg\}\\
&\le C\exp\bigg(-\frac{u^2}{3}\bigg).
\end{align*}
Putting the two terms together, we get the stated bound.
\end{proof}
We now adjust the Lemma \ref{lem:weakcontrolzijplus} to suit for proving Theorem \ref{thm:tmnplus}, in which we need to deal with
\beq\label{eq:tzij}
\tilde Z_{i, j}^+ := \frac{S_j^+ - S_i^+}{\sqrt{j - i - (S_j^+ - S_i^+)}}.
\eeq
Define a function
\beq
\phi(x) = \frac{x}{\sqrt{1 - x}}, ~~x < 1,
\eeq
and thus we have
\beq
\frac{\tilde Z_{i, j}^+}{\sqrt{j - i}} = \phi\bigg(\frac{Z_{i, j}^+}{\sqrt{j - i}}\bigg).
\eeq
Since $\phi(x)$ is strictly increasing on $(-\infty, 1)$ with range $\bbR$, we write its inverse function as
\beq\label{eq:gplus}
g^+(x) := \frac{1}{2}(x\sqrt{x^2 + 4} - x^2), ~~x \in \bbR,
\eeq
which is also strictly increasing. Therefore, $\tilde Z_{i, j}^+ \ge u$ if and only if
\beq\label{eq:tzijtransform}
 Z_{i, j}^+ \ge \sqrt{j - i} \cdot g^+\bigg(\frac{a}{\sqrt{j - i}}\bigg).
\eeq
This is an important transformation which enables us to deal with $Z_{i, j}^+$ instead. We compute the Taylor expansion of $I^+(g^+(s))$ at $s = 0$,
\beq\label{eq:iplusgplustaylor}
I^+(g^+(s)) = \frac{s^2}{2} - \frac{s^3}{6} + O(s^4).
\eeq
We have
\begin{lem}\label{lem:weakcontroltzijplus}
Fix constants $B_1$, $B_2 > 0$. Then for all $x \in \bbZ$, $l,r \in \bbN$ and all $u > 0$ such that $B_1l > u^2$ and $r < B_2lu^{-2}$, we have
\beq\label{eq:weakcontroltzijplus}
\cQ(l, r, u) := \P\bigg(\max_{(i, j) \in \bbT_r(x, x + l)} \tilde Z_{i, j}^+ \ge u\bigg) \le \frac{C}{u} \exp\bigg(-\frac{u^2}{2} + \frac{cu^3}{\sqrt{l}}\bigg),
\eeq
where the constants $c, C > 0$ depend on $B_1$ and $B_2$ but do not depend on $x,l,r,u$.
\end{lem}
\begin{proof}
By the transformation \eqref{eq:tzijtransform}, translation invariance and the fact that $g^+(x)/x^2$ is strictly decreasing, 
\begin{align}
\cQ(l, r, u) & = \P\bigg(\max_{(i, j) \in \bbT_r(0, l)} \tilde Z_{i, j}^+ \ge u\bigg) \\
&= \P\bigg[\max_{0 \le k_1, k_2 \le r} \bigg\{S_{k_1}^{(1)+} + S_{k_2}^{(2)+} -  (l + k_1 + k_2) \cdot g^+\bigg(\frac{u}{\sqrt{l + k_1 + k_2}}\bigg)\bigg\} + S_{l}^+ \ge 0\bigg]\\
&\le \P\bigg[\max_{0 \le k_1, k_2 \le r} \bigg\{S_{k_1}^{(1)+} + S_{k_2}^{(2)+} \bigg\}- l \cdot g^+\bigg(\frac{u}{\sqrt{l}}\bigg) + S_{l}^+ \ge 0\bigg], \label{eq:tzijdecompose}
\end{align}
where $S_{k_1}^{(1)+}$, $S_{k_2}^{(2)+}$ are two partial sums of $X_i$ independent of each other and $S_l^+$. Define 
\beq\label{eq:vludefn}
V_{l, u} = u\bigg(u - \frac{S_l^+}{\sqrt{l - S_l^+}}\bigg).
\eeq
Thus
\beq
\frac{S_l^+}{\sqrt{l - S_l^+}} = \frac{l \cdot S_l^+/l}{\sqrt{l}\sqrt{1 - S_l^+/l}} = \sqrt{l} \cdot \phi\bigg(\frac{S_l^+}{l}\bigg) = u - \frac{V_{l, u}}{u},
\eeq
which gives
\beq
S_l^+ = l\cdot g^+\bigg(\frac{u - V_{l, u}/u}{\sqrt{l}}\bigg).
\eeq
Therefore,
\begin{align}
&\cQ(l,r,u) \\
&\le \P\bigg[\max_{0 \le k_1, k_2 \le r} \bigg\{S_{k_1}^{(1)+} + S_{k_2}^{(2)+} \bigg\}- l \cdot g^+\bigg(\frac{u}{\sqrt{l}}\bigg) + l\cdot g^+\bigg(\frac{u - V_{l, u}/u}{\sqrt{l}}\bigg) \ge 0, V_{l, u} \le 0\bigg]\\
&+\P\bigg[\max_{0 \le k_1, k_2 \le r} \bigg\{S_{k_1}^{(1)+} + S_{k_2}^{(2)+} \bigg\}- l \cdot g^+\bigg(\frac{u}{\sqrt{l}}\bigg) + l\cdot g^+\bigg(\frac{u - V_{l, u}/u}{\sqrt{l}}\bigg) \ge 0, V_{l, u} > 0\bigg]\\
&= \P(V_{l, u} \le 0) \\
&+\P\bigg[\max_{0 \le k_1, k_2 \le r} \bigg\{S_{k_1}^{(1)+} + S_{k_2}^{(2)+} \bigg\}- l \cdot g^+\bigg(\frac{u}{\sqrt{l}}\bigg) + l\cdot g^+\bigg(\frac{u - V_{l, u}/u}{\sqrt{l}}\bigg) \ge 0, V_{l, u} > 0\bigg]\\
&= F_{l,u}(0)+ \int_0^\infty G_{l,r,u}(s)dF_{l,u}(s),\label{eq:tqlrudecompose}
\end{align}
where the last equality is obtained by conditioning on $V_{l,u} = s$, which is independent of $S_{k_1}^{(1)+}$, $S_{k_2}^{(2)+}$. $F_{l, u}$ therein is the probability distribution of $V_{l, u}$ and
\begin{align*}
G_{l, r, u}(s) :=& \P\bigg[\max_{0 \le k_1, k_2 \le r} \bigg\{ S_{k_1}^{(1)+} + S_{k_2}^{(2)+} \bigg\}- l \cdot g^+\bigg(\frac{u}{\sqrt{l}}\bigg)+ l\cdot g^+\bigg(\frac{u - s/u}{\sqrt{l}}\bigg) \ge 0\bigg],
\end{align*}
which is decreasing. To obtain an upper bound for $\cQ(l,r,u)$, first we bound $F_{l,u}(s)$ for $s \in [0, \frac{3}{4}u^2]$ so that $u - s/u \in [u/4, u]$. Applying \eqref{eq:cumulantineq},
\begin{align*}
F_{l,u}(s) &= \P\bigg(\frac{S_l^+}{\sqrt{l - S_l^+}} \ge u - \frac{s}{u}\bigg)\\
&= \P\bigg\{\frac{S_l^+}{\sqrt{l}} \ge \sqrt{l} \cdot g^+\bigg(\frac{u - s/u}{\sqrt{l}}\bigg)\bigg\}\\
& \le C\bigg\{\sqrt{l}\cdot g^+\bigg(\frac{u - s/u}{\sqrt{l}}\bigg)\bigg\}^{-1}\exp\bigg[-l\cdot I^+\bigg\{g^+\bigg(\frac{u - s/u}{\sqrt{l}}\bigg)\bigg\}\bigg]\\
& \le \frac{C}{u}\exp\bigg[-l\cdot I^+\bigg\{g^+\bigg(\frac{u - s/u}{\sqrt{l}}\bigg)\bigg\}\bigg],
\end{align*}
where the last inequality follows from the fact that when $0 < x < 1$,
\beq
xg^+\bigg(\frac{1}{x}\bigg) > \frac{1}{2}.
\eeq
By Taylor expansion of $I^+(g^+(s))$, we have
\begin{align}
F_{l, u}(s) &\le \frac{C}{u}\exp\bigg\{-\frac{1}{2}\bigg(u - \frac{s}{u}\bigg)^2 +  \frac{c}{2\sqrt{l}}\bigg(u - \frac{s}{u}\bigg)^3\bigg\}\nonumber\\
&\le \frac{Ce^s}{u}\exp\bigg(-\frac{u^2}{2} + \frac{c u^3}{\sqrt{l}}\bigg).\label{eq:tflucontrol}
\end{align}
It is however easy to see that this inequality continues to hold for $s \ge \frac{3}{4}u^2$. Indeed, if $c$ is sufficiently small, then the assumption $B_1l > u^2$ implies that $cu^3/\sqrt{l} \le u^2/8$. Hence, when $s \ge \frac{3}{4}u^2$, the above inequality becomes
\beq
F_{l, u}(s) \le \frac{C}{u}\exp\bigg(\frac{3u^2}{8}\bigg).
\eeq
If $C$ is sufficiently large, the right-hand side of previous inequality is greater than $1$ and hence the inequality trivially holds. We bound $G_{l,r,u}(s)$ for $s \ge 0$,
\begin{align*}
G_{l,r,u}(s) 
&\le \P\bigg\{\max_{0 \le k_1, k_2 < r} S_{k_1}^{(1)+} + S_{k_2}^{(2)+} > \frac{s}{2u}\sqrt{\bigg(u - \frac{s}{u}\bigg)^2 + 4l} + \frac{s^2}{2u^2} - s\bigg\}\\
&\le 2\P\bigg\{\max_{0 \le k < r} S_{k}^+ > \frac{s}{4u}\sqrt{\bigg(u - \frac{s}{u}\bigg)^2 + 4l} - \frac{s}{2}\bigg\}\\
&\le 2\P\bigg\{\max_{0 \le k < r} S_{k}^+ > \frac{s}{2u}\sqrt{l} - \frac{s}{2}\bigg\}. 
\end{align*}
Applying the Lemma \ref{lem:petrov} to the above equation we obtain
\begin{align*}
G_{l,r,u}(s) &\le 4\P\bigg(S_r^+ > \frac{s}{2u}\sqrt{l} -\frac{s}{2} - \sqrt{2r}\bigg)\\
&\le 4\P\bigg(\frac{S_r^+}{\sqrt{r}} > \frac{s}{2u\sqrt{r}}\sqrt{l}- \frac{s}{2\sqrt{r}} - \sqrt{2}\bigg)\\
&\le 4\exp\bigg\{-rI^+\bigg(\frac{cs - \sqrt{2}}{\sqrt{r}}\bigg)\bigg\}.
\end{align*}
In the second inequality, we used the assumption $r < B_2lu^{-2}$. By noticing the fact that $I^+(s) \ge s^2/2$, we have
\beq\label{eq:tglrucontrol}
G_{l,r,u}(s) \le Ce^{-cs^2}.
\eeq
Strictly speaking, this is valid only as long as $cs \ge \sqrt{2}$, however, we can choose the constant $C$ so large that \eqref{eq:tglrucontrol} continues to hold in the case $cs < \sqrt{2}$. To obtain \eqref{eq:weakcontroltzijplus}, by \eqref{eq:tqlrudecompose}, \eqref{eq:tflucontrol}, \eqref{eq:tglrucontrol}, it is clear that
\begin{align*}
\cQ(l, r, u) &\le F_{l,u}(0) + \sum_{k = 0}^\infty G_{l, r, u}(k)F_{l, u}(k + 1)\\
&\le \frac{C}{u}\bigg(1 + \sum_{k = 0}^\infty e^{-ck^2}e^k\bigg)\exp\bigg(-\frac{u^2}{2} + \frac{cu^3}{\sqrt{l}}\bigg)\\
&\le \frac{C}{u}\exp\bigg(-\frac{u^2}{2} + \frac{cu^3}{\sqrt{l}}\bigg).
\end{align*}
\end{proof}

\subsection{Proof of \thmref{mnplus} and \thmref{mnminus}}
The roadmap of our proof.
We know that $(U_{(1)}, U_{(2)}, \dots, U_{(n)})$ has the same distribution as 
\beq
\Big(\frac{Y_1}{\sum_{i = 1}^{n + 1}Y_i}, \frac{Y_1 + Y_2}{\sum_{i = 1}^{n + 1}Y_i}, \dots, \frac{\sum_{i = 1}^n Y_i}{\sum_{i = 1}^{n + 1}Y_i}\Big), \quad \text{where $Y_1, \dots, Y_{n+1}$ are iid exponential.}
\eeq 
In particular, $Y_i$ can be set as $1 - X_i$. We use this fact, together with a comparison of $\sum_{i = 1}^{n + 1}Y_i$ with its mean using a central limit theorem, to deal with the dependency among order statistics above, effectively reducing the problem to partial sums of iid random variables.
We then divide the intervals into smaller intervals, which end up contributing the most to the maximum, and larger intervals, whose contribution we show to be negligible. Although $\ui$ and $Y_i$ may be defined on different probability spaces with different probability measure, we may switch between them when there is no confusion. 
Because we only prove convergence in distribution, from now on, we put $\uj = \sum_{i = 1}^jY_i/ \sum_{i = 1}^{n + 1}Y_i$ throughout the proof.
\subsubsection{Proof of (\ref{result:mnplus})}
We study the asymptotic behavior of the statistic based on different regions of $j - i$. For  $b > 0$, define the event
\begin{align*}
A_{i, j}^{n+}(b) & = \bigg\{\frac{j - i - n(\uj - \ui)}{\sqrt{(j - i)(1 - \frac{j - i}{n})}} \le b\bigg\}\\
&=\bigg\{\uj - \ui \ge \frac{j - i}{n} - \frac{b}{\sqrt{n}} w_{i, j}^n \bigg\},
\end{align*}
where
\beq\label{w}
w_{i, j}^n := \sqrt{\frac{j - i}{n}\bigg(1 - \frac{j - i}{n}\bigg)}.
\eeq
Under this notation, we have
\beq
\big\{M_n^+ \le b\big\} = \bigcap_{0 \le i < j \le n}A^{n+}_{i, j}(b).
\eeq

Define 
\beq
u_n(\tau)  =  \bigg(1 + \frac{-3\log\log n + 2\tau}{4 \log n}  \bigg) \sqrt{2 \log n}.
\eeq 
Throughout the proof, we abbreviate $u_n(\tau)$ as $u_n$ with $\tau$ fixed.  With this choice, we have $u_n \sim \sqrt{2\log n}$.

\paragraph{Step 1: Upper bound} 
For the upper bound, it suffices to focus on the optimal range so that the maximum is achieved.  This turns out to be at $j - i \propto (\log n)^3$, as discussed below. 

Define the events
\beq\label{eq:omegancn}
\Omega_n =  \big\{|S_{n + 1}^+| \le (\log\log n) \sqrt{n} \big\}.
\eeq
By the central limit theorem, 
\beq\label{eq:omegan}
\P(\Omega_n) \to 1 ~\mbox{ as }~ n\to \infty. 
\eeq
When $j - i \le \frac{n}{\log n\log\log n}$,
\begin{align*}
&A_{i, j}^{n+}(u_n) \\
&\subseteq \Omega_n^\comp\bigcup \{ \Omega_n \bigcap A_{i, j}^{n+}(u_n)\} \\
&= \Omega_n^\comp\bigcup\bigg(\Omega_n \bigcap \bigg\{\frac{j - i - S_j^+ + S_i^+}{n + 1 - S_{n + 1}^+}\ge \frac{j - i}{n} - \frac{u_n}{\sqrt{n}}w_{i,j}^n\bigg\}\bigg)\\
&\subseteq \Omega_n^\comp\bigcup\bigg(\Omega_n \bigcap\bigg\{S_j^+ - S_i^+ \le (j - i)\frac{-1 + S_{n + 1}^+}{n} + \frac{u_n}{\sqrt{n}}(n + 1 - S_{n + 1}^+)w_{i,j}^n\bigg\}\bigg)\\
&\subseteq \Omega_n^\comp\bigcup \bigg\{S_j^+ - S_i^+ \le (j - i)\frac{\log\log n}{\sqrt{n}} + \frac{u_n}{\sqrt{n}}(n + 1 + (\log\log n)\sqrt{n})w_{i,j}^n\bigg\}\\
&= \Omega_n^\comp\bigcup\bigg\{Z_{i,j}^+ \le \frac{(\log\log n)}{\sqrt{n}}\sqrt{j - i} + u_n\cdot\bigg(1 + \frac{(\log\log n)\sqrt{n} + 1}{n}\bigg)\sqrt{1 - \frac{j - i}{n}}\bigg\}\\
&\subseteq \Omega_n^\comp\bigcup\bigg\{Z_{i,j}^+ \le \sqrt{\frac{\log\log n}{\log n}} + u_n\cdot\bigg(1 + \frac{(\log\log n)\sqrt{n} + 1}{n}\bigg)\bigg\}\\
&\subseteq \Omega_n^\comp\bigcup\{Z_{i,j}^+ \le u_n(\tau + \eps)\},
\end{align*}
for any fixed $\eps > 0$ provided that $n$ is large enough. To deal with the standardized sums $Z_{i,j}^+$, we need Theorem 1.1 and Theorem 1.2 in \citep{kabluchko2014limiting}. Because $X_1 \leq 1$, it belongs to the superlogarithm family defined in \citep{kabluchko2014limiting}. Applying Theorem 1.1 and Theorem 1.2 in \citep{kabluchko2014limiting}, we obtain
\beq\label{eq:lefttail}
\lim_{n \to \infty}\P\{ Z_n^+ \le u_n \} = \exp\bigg\{-\frac{8}{9\sqrt{\pi}}e^{-\tau}\bigg\},
\eeq
and 
\beq\label{eq:lefttail_bounds}
\lim_{A \to \infty}\liminf_{n \to \infty}\P \{ Z_n^+  = Z_n^+(A^{-1}(\log n)^3, A(\log n)^3) \} = 1.
\eeq
By \eqref{eq:omegan}, \eqref{eq:lefttail} and the fact that $(\log n)^3 \ll \frac{n}{\log n(\log\log n)}$, 
\begin{align*}
&\limsup_{n \to \infty}\P(M_n^+ \le u_n ) \\
&= \limsup_{n \to \infty}\P\bigg\{ \bigcap_{0 \le i < j \le n}A_{i, j}^{n+}(u_n)\bigg\} \\
&\le \limsup_{n \to \infty}\P\bigg\{ \bigcap_{\substack{0 \le i < j \le n :  j - i \le \frac{n}{\log n \log\log n}}}A^{n+}_{i, j}(u_n(\tau + \eps))\bigg\} + \limsup_{n \to \infty} \P(\Omega_n^\comp)\\
&\le \limsup_{n \to \infty}\P\bigg\{Z_n^+\bigg(1, \frac{n}{\log n \log\log n}\bigg) \le u_n(\tau + \eps)\bigg\} + \limsup_{n \to \infty} \P(\Omega_n^\comp)\\
&=\exp\bigg\{-\frac{8}{9\sqrt{\pi}}e^{-\tau - \eps}\bigg\}.
\end{align*}
As $\eps > 0$ is arbitrary we get
\beq
\limsup_{n \to \infty}\P(M_n^+ \le u_n ) \le \lim_{\eps \to 0} \exp\bigg\{-\frac{8}{9\sqrt{\pi}}e^{-\tau - \eps}\bigg\} = \exp\bigg\{-\frac{8}{9\sqrt{\pi}}e^{-\tau}\bigg\}.
\eeq
\paragraph{Step 2: Lower bound} 
Define
\beq
k_n = \frac{n}{\log n(\log\log n)}, \quad
K_n = \frac{n\log\log n}{\log n}.
\eeq 
We establish the lower bound by dividing the range of $j - i$ into five regions: 
\begin{align*}
R_1 &= [1, u_n^2), &  
R_2 &= [u_n^2, k_n), \\
R_3 &= [k_n, K_n), &
R_4 &= [K_n, n - K_n), \\ 
R_5 &= [n - K_n, n). &
\end{align*} 

\medskip
\noindent
$\bullet$ For $R_1$, note that 
\beq
\frac{j - i}{n} - \frac{u_n}{\sqrt{n}}w_{i, j}^n \le 0,
\eeq
is equivalent to
\beq
j - i \le \frac{u_n^2}{1 + u_n^2/n}.
\eeq
Since $u_n^4 \ll n$, $i, j$ only take value in integers, it is further equivalent to $j - i \le u_n^2$ when $n$ is large enough, which is exactly $R_1$. 
Therefore, when $n$ is large enough, 
\beq
A^{n+}_{i, j}(u_n) = \Omega,
\eeq
for any $(i,j)$ satisfying $j-i \in R_1$ so that
\beq
\bigcap_{0 \le i < j \le n: j - i \in R_1}A^{n+}_{i, j}(u_n) = \Omega.
\eeq

\noindent
$\bullet$  For $R_2$, following the same argument that was used to prove the upper bound, it can be shown that 
\beq
\liminf_{n \to \infty}\P\bigg\{ \bigcap_{\substack{0 \le i < j \le n :  j - i \in  R_2}}A^{n+}_{i, j}(u_n)\bigg\} \ge \exp\bigg\{-\frac{8}{9\sqrt{\pi}}e^{-\tau}\bigg\}.
\eeq

\noindent
$\bullet$ Turning to $R_3$, we shall show that 
\beq
\P\bigg(\max_{0 \le i \le n - k_n} \frac{S_{i + k_n}^+ - S_i^+}{\sqrt{k_n}} \le \log\log n \bigg) \to 1,
\eeq
and then use this fact to prove that the maximum of $M_{i, j}^+$ over $R_3$ is ignorable. First we bound $\max_{0 \le i \le n - k_n} (S_{i + k_n}^+ - S_i^+)$. Define 
\beq
q_n  = \frac{k_n}{(\log\log n)^2} \ll k_n,
\eeq
and introduce a positive sequence $\eps_n$ such that $q_n \ll \eps_n \ll k_n$. Consider the following two-dimensional grid with mesh size $q_n$:
\beq
\cJ_n = \{ (x, y) \in q_n\bbZ^2 : x \in [-\epsilon_n, n + \epsilon_n], y - x \in [0.9k_n - \epsilon_n, 1.1k_n + \epsilon_n] \}.
\eeq
By the union bound,
\beq
\P \big\{ Z_n^+(0.9k_n, 1.1k_n) > \log\log n \big\} \le \sum_{(x, y) \in \cJ_n}\P\bigg\{  \max_{(i, j) \in \bbT_{q_n}(x, y)} Z_{i, j}^+ \ge \log\log n  \bigg\}.
\eeq
Note that the cardinality of $\cJ_n$ satisfies
\beq
|\cJ_n| \sim \frac{(1.1 - 0.9)nk_n}{(q_n)^2} = 0.2(\log\log n)^5 \log n .
\eeq
By the translation invariance property of $\bbT_{q_n}(x, y)$ and Lemma \ref{lem:weakcontrolzijplus}, taking $l = y - x$, $r = q_n$ and $u = \log\log n$ for large enough $n$ (and thus satisfying the conditions in Lemma \ref{lem:weakcontrolzijplus}) temporarily, we have
\begin{align*}
\P\big\{ Z_n^+(0.9k_n, 1.1k_n) \ge \log\log n \big\} &\le C|\cJ_n|\exp\bigg\{-\frac{(\log\log n)^2}{2} \bigg\} \to 0,
\end{align*}
where $C > 0$ is a constant. Since 
\beq
\max_{0 \le i \le n - k_n} \frac{S_{i + k_n}^+ - S_i^+}{\sqrt{k_n}}\le Z_n^+(0.9k_n, 1.1k_n),
\eeq
it follows that
\beq\label{eq:fixscan}
\limsup_{n \to \infty}\P\bigg( \max_{0 \le i \le n - k_n} \frac{S_{i + k_n}^+ - S_i^+}{\sqrt{k_n}} \ge \log\log n \bigg) = 0.
\eeq
We may now prove the ignorability of maximum of $M_{i, j}^+$ when taking values on $R_3$. Define
\beq
\Omega_{1n} := \bigg\{\max_{0 \le i \le n - k_n} \frac{S_{i + k_n}^+ - S_i^+}{\sqrt{k_n}} \le \log\log n\bigg\}.
\eeq
By \eqref{eq:fixscan}, $\P(\Omega_{1n}) \to 1$ as $n \to \infty$. For $j - i \in R_3$, 
\begin{align*}
&A_{i, j}^{n+}(u_n) \\
&\supseteq \Omega_n\bigcap\Omega_{1n}\bigcap \bigg\{S_j^+ - S_i^+ \le (j - i)\frac{S_{n + 1}^+ - 1}{n} + \frac{u_n}{\sqrt{n}}(n + 1 - S_{n + 1})w_{i,j}^n\bigg\}\\
&=\Omega_n\bigcap\Omega_{1n}\bigcap \bigg\{S_j^+ - S_{i + k_n}^+ \le (j - i)\frac{S_{n + 1}^+ - 1}{n} - S_{i + k_n}^+ + S_i^+ + \frac{u_n}{\sqrt{n}}(n + 1 - S_{n + 1}^+)w_{i,j}^n\bigg\}\\
&\supseteq\Omega_n\bigcap\Omega_{1n}\bigcap \bigg\{S_j^+ - S_{i + k_n}^+ \le -(j - i)\frac{\log \log n}{\sqrt{n}} - \sqrt{k_n}\log\log n + \frac{u_n}{\sqrt{n}}(n + 1 - \log\log n\sqrt{n})w_{i,j}^n\bigg\}\\
&\supseteq\Omega_n\bigcap\Omega_{1n}\bigcap \bigg\{\frac{S_j^+ - S_{i + k_n}^+}{\sqrt{j - i - k_n}} \le\sqrt{\frac{j - i}{j - i - k_n}}\bigg[u_n\cdot\bigg(  1 - \frac{\log \log n}{\sqrt{n}}  \bigg) - \sqrt{\frac{(\log\log n)^3}{\log n}}-\log\log n\bigg]\bigg\}\\
&\supseteq\Omega_n\bigcap\Omega_{1n}\bigcap \bigg\{\frac{S_j^+ - S_{i + k_n}^+}{\sqrt{j - i - k_n}} \le\sqrt{1 + \frac{k_n}{K_n}}\bigg[u_n\cdot\bigg(1 -\frac{\log \log n}{\sqrt{n}}  \bigg) - \sqrt{\frac{(\log\log n)^3}{\log n}} -\log\log n\bigg]\bigg\}\\
&\supseteq\Omega_n\bigcap\Omega_{1n}\bigcap \bigg\{\frac{S_j^+ - S_{i + k_n}^+}{\sqrt{j - i - k_n}} \le u_n(\log\log n)\bigg\},
\end{align*}
where the last line follows by noting that $k_n/K_n = 1/(\log\log n)^2$. Thus 
\begin{align*}
&\bigcap_{0 \le i < j \le n: ~k_n + 1 \le j - i \le K_n}A_{i, j}^{n+}(u_n) \\
&\supset \Omega_n\bigcap\Omega_{1n}\bigcap \bigg\{\max_{0 \le i < j \le n: ~k_n + 1 \le j - i \le K_n}\frac{S_j^+ - S_{i + k_n}^+}{\sqrt{j - i -k_n}} \le u_n(\log\log n)\bigg\}\\
&\supset \Omega_n\bigcap\Omega_{1n}\bigcap \bigg\{\max_{0 \le i < j \le n: ~j - i \le K_n}\frac{S_j^+ - S_i^+}{\sqrt{j - i}} \le u_n(\log\log n)\bigg\},
\end{align*}
and recall that $u_n(\cdot)$ is a function. Since $(\log n)^3 \ll K_n$, \eqref{eq:lefttail} and \eqref{eq:lefttail_bounds} together imply that 
\begin{align*}
&\liminf_{n \to \infty}\P\big\{ M_n^+(k_n + 1, K_n) \le u_n(\tau) \big\} \\
&\ge \liminf_{n \to \infty}\P\bigg[\Omega_n\bigcap\Omega_{1n}\bigcap  \{Z_n^+(1, K_n) \le u_n(\log\log n) \}\bigg] \\
&\ge \liminf_{n \to \infty}\P\bigg[\Omega_n\bigcap\Omega_{1n}\bigcap  \{Z_n^+(1, K_n) \le u_n(\tau') \}\bigg]= \exp\bigg(-\frac{8}{9\sqrt{\pi}}e^{-\tau'}\bigg),
\end{align*}
for any $\tau, \tau'$. We now take $\tau' \to \infty$, yielding
\beq
\liminf_{n \to \infty}\P\big\{ M_n^+(k_n + 1, K_n) \le u_n(\tau) \big\} = \liminf_{\tau' \to \infty}\exp\bigg( -\frac{8}{9\sqrt{\pi}}e^{-\tau'} \bigg)= 1.
\eeq
\noindent
$\bullet$ Next we apply the Kolmogorov's Theorem to deal with $R_4$.
Define the centered order statistics
\beq
\bui = \ui - \frac{i}{n + 1}. 
\eeq
Note that when $n$ is large enough,
\begin{align*}
A_{i, j}^{n+}(u_n) &= \bigg\{\buj - \bui \ge \frac{j - i}{n(n + 1)} - \frac{u_n}{\sqrt{n}}w_{i, j}^n \bigg\}\\
&= \bigg\{\sqrt{n}(\buj - \bui) \ge \frac{j - i}{\sqrt{n}(n + 1)} - u_nw_{i, j}^n \bigg\}\\
&\supseteq \bigg\{\sqrt{n}(\buj - \bui) \ge - 0.9u_nw_{i, j}^n \bigg\}\\
&\supseteq \bigg\{0.9u_nw_{i, j}^n \ge \sqrt{n}(\buj - \bui) \ge - 0.9u_nw_{i, j}^n \bigg\}\\
&\supseteq \bigg\{2\sqrt{n} \max\{|\bui|, |\buj|\} \le 0.9u_nw_{i, j}^n \bigg\}.
\end{align*}
For $(i,j)$ such that $j-i \in R_4$, $w_{i, j}^n$ is minimized at either $j - i = \frac{n\log\log n}{\log n}$ or $n - \frac{n\log\log n}{\log n}$.
Consequently,
\begin{align*}
\bigcap_{0 \le i < j \le n: j - i \in R_4}A^n_{i, j}(u_n) &\supseteq \bigg\{\sqrt{n} \max_{1 \le i \le n}\{|\bui|\} \le \frac{0.9u_n}{2}\min_{0 \le i < j \le n: j - i \in R_4}w_{i, j}^n\bigg\}\\
&=\bigg\{\sqrt{n} \max_{1 \le i \le n}\{|\bui|\} \le \frac{0.9u_n}{2}\sqrt{\frac{\log\log n}{\log n}\bigg(1 - \frac{\log\log n}{\log n} \bigg)}\bigg\}.
\end{align*}
The Kolmogorov's Theorem states that for any $y\ge0$,
\beq
\lim_{n\to\infty} \P\bigg(\sqrt{n}\max_{1 \le i \le n}|\bui| \le y\bigg) = K(y) := 1 - 2e^{-2y^2} +2e^{-8y^2} - \cdots~.
\eeq
In particular, $(\sqrt{n}\max_{1 \le i \le n}|\bui|)$ is tight. Therefore, by the fact that
\beq
\frac{0.9u_n}{2}\sqrt{\frac{\log\log n}{\log n}\bigg(1 - \frac{\log\log n}{\log n}\bigg)} \asymp \sqrt{\log\log n} \to \infty,
\eeq 
we obtain
\beq\label{eq:kolplus}
\lim_{n \to \infty}\P\bigg\{ \bigcap_{0 \le i < j \le n: j - i \in R_4}A_{i, j}^{n+}(u_n)\bigg\} = 1.
\eeq

\noindent
$\bullet$ For $R_5$, define $j' = n - j$ and $U_{(j' + 1)}' = 1 - U_{(n + 1- j' - 1)} = 1 - U_{(j)}$. A simple change of indices gives 
\begin{align*}
M_n^+(n - K_n, n) 
&= \max_{\substack{0 \le i < j \le n\\n - K_n \le j - i < n}}\frac{j - i - n(\uj - \ui)}{\sqrt{(j - i)(1 - \frac{j - i}{n})}}\\
&\le \max_{\substack{i, j' \ge 0\\  i + j' < K_n}} \frac{nU_{(j' + 1)}' - (j' + 1) + n\ui - i}{\sqrt{(i + j') (1 - \frac{i + j'}{n} )}} \\
&\le 1.01\max_{\substack{i, j\ge 0\\ 1 \le i + j < K_n}} \frac{n\uj' - j + n\ui - i}{\sqrt{i + j}} + 1.01 \\
\end{align*}
where the last inequality holds when $n$ is large enough since $K_n \ll n$. 
Now, by the above statements, to prove 
\beq
\limsup_{n \to \infty}\P(M_n^+(n - K_n, n) \ge u_n) = 0,
\eeq
it suffices to prove
\beq
\limsup_{n \to \infty}\P\bigg(\max_{\substack{i, j \ge 0\\ 1 \le i + j < K_n}} \frac{n\ui - i + n\uj' - j}{\sqrt{i + j}} \ge \sqrt{1.9\log n}\bigg) = 0.
\eeq
Assuming $0/0 = 0$, observe that
\begin{align*}
&\P\bigg(\max_{\substack{i, j \ge 0\\ 1 \le i + j < K_n}} \frac{n\ui - i + n\uj' - j}{\sqrt{i + j}} \ge \sqrt{1.9\log n}\bigg)\\
&= \P\bigg\{\max_{\substack{i, j \ge 0\\1 \le i + j \le K_n}} \bigg(\frac{n\ui - i}{\sqrt{i + j}} + \frac{n\uj' - j}{\sqrt{i + j}} \bigg) \ge \sqrt{1.9\log n}\bigg\}\\
&\le \P\bigg\{\max_{\substack{i, j \ge 0\\1 \le i + j \le K_n}} \bigg(\frac{n\ui - i}{\sqrt{i}}+ \frac{n\uj' - j}{\sqrt{j}}\bigg) \ge \sqrt{1.9\log n}\bigg\}\\
&\le \P\bigg(\max_{0 \le i \le n} \frac{n\ui - i}{\sqrt{i}} + \max_{0 \le j \le n}\frac{n\uj' - j}{\sqrt{j}}  \ge \sqrt{1.9\log n}\bigg)\\
&\le 2\P\bigg(\max_{0 \le i \le n} \frac{n\ui - i}{\sqrt{i}}  \ge \frac{\sqrt{1.9\log n}}{2}\bigg)\\
&\le 2\P\bigg(\max_{0 \le i \le n} \frac{n\ui - i}{\sqrt{i(1 - i/n})}  \ge \frac{\sqrt{1.9\log n}}{2}\bigg).
\end{align*}
However, \citet{eicker1979asymptotic} showed that
\beq
\max_{0 \le i \le n} \frac{n\ui - i}{\sqrt{i(1 - i/n)}} \sim \sqrt{2\log\log n},
\eeq
which finishes the proof for $R_5$.

\noindent
$\bullet$ Now combining all the results gives the lower bound, which, together with the upper bound, establishes the proof of \thmref{mnplus}.
\qed

\subsubsection{Proof of (\ref{result:mnminus})}
In what follows, we let 
\beq
u_n = u_n(\tau) := \log n + \tau, 
\eeq
with $\tau$ fixed. 
Define
\begin{align*}
A_{i, j}^{n-}(u_n) & = \bigg\{\frac{n(\uj - \ui) - (j - i)}{\sqrt{(j - i)(1 - \frac{j - i}{n})}} \le u_n\bigg\} =\bigg\{\uj - \ui \le \frac{j - i}{n} + \frac{u_n}{\sqrt{n}}w_{i, j}^n \bigg\},
\end{align*}
where $w_{i, j}^n$ is defined in \eqref{w}, and note that  
\beq
\big\{ M_n^- \le u_n \big\} = \bigcap_{0 \le i < j \le n} A_{i,j}^{n-}(u_n).
\eeq 

\paragraph{Step 1: Upper bound} 
For the upper bound, again, we only consider a particular order of magnitude for the length, the one that contributes the most to the maximum. When $j - i \le \frac{n \log\log n}{(\log n)^2}$, 
\begin{align*}
A_{i, j}^{n-}(u_n) &\subset \Omega_n^\comp \bigcup \{ \Omega_n \bigcap A^{n-}_{i, j}(u_n) \} \\
&\subset \Omega_n^\comp \bigcup \bigg\{\frac{S_j^- - S_i^-}{\sqrt{j - i}} \le (\log\log n)\sqrt{\frac{j - i}{n}} + u_n\cdot\bigg(1 + \frac{\log\log n}{\sqrt{n}}\bigg)\bigg\}\\
&\subset \Omega_n^\comp \bigcup \bigg\{\frac{S_j^- - S_i^-}{\sqrt{j - i}} \le  u_n(\tau + \eps)  \bigg\},
\end{align*}
for any $\eps > 0$, where $\Omega_n$ is given in \eqref{eq:omegancn}. By \eqref{eq:omegan}, it suffices to consider the second event on the RHS.  Applying Theorem 1.7 in \citep{kabluchko2014limiting}, the limiting distribution of $Z_n^-$ is the same as that of $\max_{1 \le i \le n} (-X_i)$. By the independence of $\{X_i\}$, we obtain
\begin{align}\label{lem:righttail}
\lim_{n \to \infty}\P( Z_n^- \le u_n ) &= \lim_{n \to \infty}\P\{ \max_{1 \le i \le n} (-X_i) \le u_n \}  = \exp\{ -\exp(1 - \tau) \}. 
\end{align}
Therefore, taking $\eps \to 0$,
\begin{align*}
\limsup_{n \to \infty}\P(M_n^- \le u_n)
&= \limsup_{n \to \infty}\P\bigg\{\bigcap_{0 \le i < j \le n}A_{i, j}^{n-}(u_n)\bigg\}\\
&\le \limsup_{n \to \infty}\P\bigg\{ \bigcap_{\substack{0 \le i < j \le n: j - i \le \frac{n\log\log n}{(\log n)^2}}}A_{i, j}^{n-}(u_n)\bigg\} + \P(\Omega_n^\comp)\\
&\le \limsup_{\eps \to 0}\exp \{ -\exp(1 - \tau - \eps) \}\\
&=\exp \{ -\exp(1 - \tau) \}.
\end{align*}



\paragraph{Step 2: Lower bound}
As in the proof of \eqref{thm:mnplus}, we divide the range of $j - i$ into several subintervals. Similar to the upper bound case, 
\beq
\lim_{n\to \infty }\P \bigg\{ M_n^-\bigg(1, \frac{n\log\log n}{(\log n)^2}\bigg)   \leq u_n \bigg\} = \exp \{ -\exp(1 - \tau) \}.
\eeq 
With the same argument that was used to prove \eqref{eq:kolplus}, we obtain
\beq
\lim_{n \to \infty}\P\bigg\{ \bigcap_{\substack{0 \le i < j \le n :  \frac{n\log\log n}{(\log n)^2} \le j - i \le n - \frac{n\log\log n}{(\log n)^2}}}A_{i, j}^{n-}(u_n)\bigg\}  = 1.
\eeq
The case where $j - i \ge n - \frac{n\log\log n}{(\log n)^2}$ can be treated similarly to proving the region $R_5$ in the proof of Theorem \ref{thm:mnplus}, even easier since now $u_n \sim \log n$ (and details are omitted).

\subsubsection{Proof of (\ref{result:mn})}
This follows directly from \eqref{result:mnplus}, where we learn that $M_n^+ \asymp_P \sqrt{\log n}$, and \eqref{result:mnminus}, which states that $M_n^- \asymp_P \log n$ (here $A_n \asymp_P B_n$ means $A_n = O(B_n)$ and $B_n = O(A_n)$). Combining them implies that $M_n^- \gg_P M_n^+$, and therefore $M_n = \max(M_n^-, M_n^+) = M_n^-$ with probability tending to 1 as $n$ increases.

\subsection{Proof of \thmref{tmnplus}}
\subsubsection{Proof of (\ref{eq:tmnplusall})}
We first derive the asymptotic distribution of 
\beq
\tilde M_n^+(1,2) = \max_{0 \le i \le n - 1}\frac{1 - n(U_{(i + 1)} - \ui)}{\sqrt{n(U_{(i + 1)} - \ui)(1 - U_{(i + 1)} + \ui)}},
\eeq
which is exactly the same as that of \eqref{eq:tmnplusall} and then show that $\tilde M_n^+(2, n) \ll_P \sqrt{n}$. These together imply \eqref{eq:tmnplusall}. To get the asymptotic distribution of $\tilde M_n^+(1,2)$, note that
\begin{align}
\tilde M_n^+(1,2)  \le \max_{0 \le i \le n - 1}\frac{1}{\sqrt{n(U_{(i + 1)} - \ui)[1 - (U_{(i + 1)} - \ui)]}}  \label{eq:tmnoneoneupper}\\
\mbox{ and }~ \tilde M_n^+(1,2) \ge \max_{0 \le i \le n - 1}\frac{1 - n(U_{(i + 1)} - \ui)}{\sqrt{n(U_{(i + 1)} - \ui)}} \label{eq:tmnoneonelower},
\end{align}
where both upper and lower bounds are functions of 
\beq
T := \min_{0 \le i \le n - 1} (U_{(i + 1)} - \ui). 
\eeq
Therefore it suffices to work on $T$ instead. It is easy to see that $T \le 1/n$. By symmetry, 
\beq
\P(T \ge t) = n!\P(T \ge t, U_1 \le U_2 \le \cdots \le U_n). 
\eeq
Define the subset
\beq
A_t  = \{(u_1, \ldots, u_n) \in [0, 1]^n: u_i + t \le u_{i + 1},  i = 0, 1, \ldots, n-1 \},
\eeq
where $u_0 = 0$. Then, 
\beq
\{(U_1,\cdots, U_n) \in A_t\} = \{T \ge t, U_1 \le U_2 \le \cdots \le U_n\},
\eeq
and hence
\beq
\P(T \ge t, U_1 \le U_2 \le \cdots \le U_n) = \lambda_n(A_t),
\eeq
where $\lambda_n$ is the Lebesgue measure on $\bbR^n$. Define a mapping
\beq
h: \quad  A_t \longrightarrow Q  \subset [0, 1 - nt]^n, \quad h(u_1,u_2,\cdots, u_n) = (u_1 - t, u_2 - 2t, u_n - nt),
\eeq
where 
\beq
Q := \{(y_1,\ldots,y_n): y_i \le y_{i + 1}, \forall ~ 1 \le i \le n - 1\} \cap [0, 1 - nt]^n. 
\eeq
It is easy to verify that $h$ is a volume-preserving bijection. Hence 
\beq
\P(T \ge t, U_1 \le U_2 \le \cdots \le U_n) = \lambda_n(A_t) = \lambda_n(Q) = \frac{(1 - nt)^n}{n!}
\eeq
Therefore, we have 
\beq
\P(T \ge t) = \frac{n!(1 - nt)^n}{n!} = (1 - nt)^n,
\eeq
for $0 \le t \le 1/n$. For any $0 \le t \le 1/n$,
\beq
\P\bigg\{  \min_{0 \le i \le n - 1} (U_{(i + 1)} - \ui) \ge t  \bigg\} = (1 - n t)^n ,
\eeq 
which implies
\beq
\lim_{n \to \infty}\P\bigg\{  \min_{0 \le i \le n - 1} (U_{(i + 1)} - \ui) \ge \frac{\tau}{n^2}  \bigg\} = \exp(-\tau).
\eeq
This, together with \eqref{eq:tmnoneoneupper} and \eqref{eq:tmnoneonelower}, implies
\beq
\lim_{n \to \infty}\P \bigg( \tilde M_n^+(1, 1) \le \sqrt{\frac{n}{\tau}} \bigg) = \exp(-\tau).
\eeq
It remains to show that $\tilde M_n^+(2 , n) \ll_P \sqrt{n}$. We will divide it into $\tilde M_n^+(2, (\log n)^2)$, $\tilde M_n^+((\log n)^2, n - (\log n)^2)$ and $\tilde M_n^+(n - (\log n)^2, n)$. When $2 \le j - i \le (\log n)^2$, note that
\begin{align}
1 - (\uj - \ui) &= 1 - \frac{j - i}{n + 1} - (\buj - \bui)\\
&\ge 1 - \frac{(\log n)^2}{n + 1} - 2\max_{1 \le i \le n} |\bui|\nonumber\\
&= 1 + O_P(1/\sqrt{n})\nonumber\\
&\ge 0.5,\label{eq:oneujui}
\end{align}
where the last inequality holds on a sequence of events with probability tending to one, by Kolmogorov's Theorem mentioned in the proof of \thmref{mnplus} when $n$ is large enough. Meanwhile,
\begin{align}
\frac{j - i - n(\uj - \ui)}{\sqrt{n(\uj - \ui)}} &= \frac{j - i - \frac{n}{n + 1 - S_{n + 1}^+}(j - i - S_j^+ + S_i^+)}{\sqrt{n\frac{n}{n + 1 - S_{n + 1}^+}(j - i - S_j^+ + S_i^+)}}\nonumber\\
&= (1 + O_P(1/\sqrt{n}))\tilde Z_{i, j} +O_P(1/\sqrt{n})\nonumber\\
&\le 1.01\tilde Z_{i, j} + 0.01,\label{eq:mnplusbound}
\end{align}
on the sequence of events $\Omega_n$ defined in \eqref{eq:omegan}. With these results, the union bound, \eqref{eq:rescumulantineq} and the fact that $I^+(s) = -s - \log(1 - s)$ on $[0, 1)$, for any $\eps > 0$, 
\begin{align*}
&\P(\tilde M_n^+(2, (\log n)^2) \ge \eps \sqrt{n}) \\
&\le \P(\tilde Z_n^+(2, (\log n)^2) \ge 0.9\eps \sqrt{n}) + \P(\Omega_n^c)\\
&\le \sum_{0 \le i < j \le n: 2 \le j - i \le (\log n)^2}\P(\tilde Z_{i, j}^+ \ge 0.9\eps \sqrt{n}) + \P(\Omega_n^c)\\
&\le n\sum_{2 \le k \le (\log n)^2}\P\bigg(\frac{S_k^+}{\sqrt{k - S_k^+}} \ge 0.9\eps \sqrt{n}) + \P(\Omega_n^c)\\
&\le n\sum_{2 \le k \le (\log n)^2}\exp\bigg[ -kI^+\bigg\{g^+\bigg(\frac{0.9\eps \sqrt{n}}{\sqrt{k}}\bigg\}\bigg]  \bigg\}+ \P(\Omega_n^c)\\
&\le n\sum_{2 \le k \le (\log n)^2}\exp\bigg[ kg^+\bigg(\frac{0.9\eps \sqrt{n}}{\sqrt{k}}\bigg) + k\log\bigg\{1 - g^+\bigg(\frac{0.9\eps \sqrt{n}}{\sqrt{k}}\bigg)\bigg\}  \bigg]+ \P(\Omega_n^c). \label{eq:tmplusgplus}
\end{align*}
As $a \to \infty$, $0.9\eps\sqrt{n}/\sqrt{k} \to \infty$ and $g^+(a) \uparrow 1$. In addition,
\beq
1 - g^+(a) = 1 - \frac{a(\sqrt{a^2 + 4} - a)}{2} = 1 - \frac{2a}{\sqrt{a^2 + 4} + a} = \frac{\sqrt{a^2 + 4} - a}{\sqrt{a^2 + 4} + a} = \frac{4}{(\sqrt{a^2 + 4} + a)^2}.
\eeq
Note that
\beq
\frac{0.9}{a^2} \le \frac{4}{(\sqrt{a^2 + 4} + a)^2} \le \frac{1}{a^2},
\eeq
when $a$ is large enough. Therefore, when $n$ is sufficiently large,
\begin{align*}
\P(\tilde M_n^+(2, (\log n)^2) \ge \eps \sqrt{n}) &\le n\sum_{2 \le k \le (\log n)^2}\exp\bigg\{ k - k\log\bigg(\frac{0.9\eps n}{k}\bigg)  \bigg\} \\
&\le n\sum_{2 \le k \le (\log n)^2}\exp(- 0.9k\log n  ) \\
&\le n\sum_{2 \le k \le (\log n)^2}\exp( -1.8\log n ) \to 0,
\end{align*}
where the last inequality uses that $k \ge 2$.

When $(\log n)^2 \le j - i \le n - (\log n)^2$, by Theorem \ref{thm:mnplus} and Theorem \ref{thm:mnminus}, we have
\beq\label{eq:orderdiffup}
\uj - \ui \le \frac{j - i}{n} +\frac{1.01\log n}{ \sqrt{n}}w_{i, j}^n,
\eeq
\beq\label{eq:orderdiffup2}
1 - (\uj - \ui) \ge 1 - \frac{j - i}{n} - \frac{1.01\log n}{ \sqrt{n}}w_{i, j}^n,
\eeq
\beq\label{eq:orderdifflow}
\uj - \ui \ge \frac{j - i}{n} - \frac{1.01\log n}{\sqrt{n}}w_{i, j}^n,
\eeq
and
\beq\label{eq:orderdifflow2}
1 - (\uj - \ui) \le 1 - \frac{j - i}{n} + \frac{1.01\log n}{\sqrt{n}}w_{i, j}^n,
\eeq
with probability tending to one. Together, \eqref{eq:orderdiffup} and \eqref{eq:orderdifflow} lead to
\beq
\bigg|\frac{n(\uj - \ui)}{j - i}\bigg| = O_P(1),
\eeq
uniformly in $(i, j)$ satisfying $j-i \ge (\log n)^2$. \eqref{eq:orderdiffup2} and \eqref{eq:orderdifflow2} imply
\beq
\bigg|\frac{1 - (\uj - \ui)}{1 -(j - i)/n}\bigg| = O_P(1).
\eeq
These, combined with the definitions of $M_n^+$ and $\tilde M_n^+$, imply
\beq
\tilde M_n^+\{(\log n)^2, n - (\log n)^2\} \asymp_P M_n^+\{(\log n)^2, n - (\log n)^2\}.
\eeq
By Theorem \ref{thm:mnplus}, it follows that for any $\eps > 0$,
\beq
\lim_{n \to \infty}\P [ \tilde M_n^+\{(\log n)^2, n - (\log n)^2\} \ge \eps\sqrt{n} ]  = 0.
\eeq
Finally, when $n - (\log n)^2 \le j - i \le n$,  define $j' = n - j$ and thus $U_{(j' + 1)}' = 1 - U_{(n + 1- j' - 1)} = 1 - U_{(j)}$. A simple change of indices gives 
\begin{align*}
&\tilde M_n^+(n - (\log n)^2, n)\\
&= \max_{\substack{0 \le i < j \le n\\n - (\log n)^2 \le j - i \le n}}\frac{j - i - n(\uj - \ui)}{\sqrt{n(\uj - \ui)(1 - (\uj - \ui))}}\\
&=\max_{\substack{i, j' \ge 0\\ i + j' \le (\log n)^2}} \frac{nU_{(j' + 1)}' - (j' + 1) + n\ui - i}{\sqrt{n(\ui + U_{(j' + 1)}') (1 - \ui - U_{(j' + 1)}')}}\\
&=\max_{\substack{i, j \ge 0\\ 1\le i + j \le (\log n)^2}} \frac{n\ui - i + n\uj' - j}{\sqrt{n(\ui + \uj') (1 - \ui - \uj')}} + O_P(1).
\end{align*}
Notice that when $i, j \ge 0$ and $1 \le i + j \le (\log n)^2$,
\beq
1 - \ui - \uj' > 1 -2 \max_{0 \le i \le (\log n)^2}\ui > 0.5,
\eeq
with probability tending to one, which can be seen by a simple application of Kolmogorov's Theorem. By a similar speech when proving $R_5$ in the proof of Theorem \ref{thm:mnplus}, 
\begin{align}
&\P\bigg(\max_{\substack{i, j \ge 0\\ 1 \le i + j \le (\log n)^2}} \frac{n\ui - i + n\uj' - j}{\sqrt{n(\ui +\uj') (1 - \ui - \uj')}}\ge \eps \sqrt{n} \bigg) \\
&\le \P\bigg(\max_{\substack{i, j \ge 0\\1 \le i + j \le (\log n)^2}} \frac{n\ui - i + n\uj' - j}{\sqrt{n(\ui +\uj')}}\ge 0.5 \eps \sqrt{n} \bigg) \\
&\le 2\P\bigg(\max_{0 \le i \le (\log n)^2} \frac{n\ui - i}{\sqrt{n\ui}}\ge 0.25\eps  \sqrt{n} \bigg) \\
&\le 2\P\bigg(\max_{0 \le i \le (\log n)^2}  \frac{n\ui - i}{\sqrt{n\ui(1 - \ui)}}\ge 0.25\eps \sqrt{n}  \bigg) \\
&\to 0, \label{eq:tmnpluslast}
\end{align}
where the last line again follows from \citet{eicker1979asymptotic}. These eventually establish the proof of \eqref{eq:tmnplusall}.

\subsubsection{Proof of (\ref{eq:tmnpluscubelocal})}
{\bf The roadmap of our proof. }

To derive the asymptotic distribution, we first focus on the most contributed part, i.e., those with length $j - i = l_n \sim a\log^3 n$ for $a > 0$. Define 
\beq
u_n = u_n(\tau) := \sqrt{2\log n}\bigg(1 + \frac{-3\log\log n + 2\tau}{4\log n}\bigg).
\eeq
For any two constants $0 < A_1 < A_2 < \infty$, define $l_n^-  = A_1 \log^3 n$ and $l_n^+ = A_2 \log^3 n$. We prove
\beq\label{eq:tmnplusrestricted}
\lim_{n \to \infty}\P\{  \tilde M_n^+(l_n^-, l_n^+) \le u_n\} = \exp\bigg\{ -e^{-\tau}\int_{A_1}^{A_2} \Lambda_1(a)da\bigg\}.
\eeq
It turns out that to prove \eqref{eq:tmnplusrestricted}, within that region, it suffices to focus on
\beq
\tilde Z_{i, j}^+ := \frac{S_j^+ - S_i^+}{\sqrt{j - i - (S_j^+ - S_i^+)}},
\eeq
instead, up to restricting on subset $\Omega_n$ defined in \eqref{eq:omegan}. Write
\beq
\tilde Z_n^+(k, l) = \max_{0 \le i < j \le n: k \le j - i \le l } \tilde Z_{i, j}^+,
\eeq
and
\beq
\tilde Z_n^+ = \tilde Z_n^+(1, n). 
\eeq
We will use Lemma \ref{lem:interchange} to show that
\beq\label{eq:local2}
\cQ_n := \P\bigg(\max_{(i, j) \in \bbT_{Bq_n}(x, x + l_n)} \tilde Z_{i, j}^+ \ge u_n\bigg) \sim P_n(0)\bigg\{  1 + H^2\bigg(\frac{B}{a}\bigg)  \bigg\},
\eeq
where $B\geq 1$ is an integer and the quantities $P_n(0)$, $H(x)$, $q_n$ will be specified later. Next, with a domain $\bbJ_n(z)$ (to be specified) larger than $\bbT_{Bq_n}$, we will show that
\beq\label{eq:largelocalrate}
\P\bigg(\max_{(i, j) \in \bbJ_n(z)} \tilde Z_{i, j}^+ \ge u_n \bigg) \sim e^{-\tau}\frac{w_n}{n}\int_{A_1}^{A_2} \Lambda_1(a)da,
\eeq
which no longer depends on $B$, with $\Lambda_1(a)$ defined in the theorem part. This enables us to apply Poisson limit theorem in \cite{arratia1989two} to get
\beq\label{eq:tznplus}
\lim_{n \to \infty}\P\{  \tilde Z_n^+(l_n^-, l_n^+) \le u_n \} = \exp\bigg\{ -e^{-\tau}\int_{A_1}^{A_2} \Lambda_1(a)da\bigg\}.
\eeq
The final step will be showing that the region beyond $A_2(\log n)^3$ is negligible, that is,
\beq\label{eq:tznplusbeyond}
\limsup_{A_2 \to \infty}\limsup_{n \to \infty}\P\{  \tilde M_n^+(l_n^+, n) \ge u_n \} = 0.
\eeq
Therefore setting $A_1 = A$ and letting $A_2 \to \infty$ yield \eqref{eq:tmnpluscubelocal}.

We first argue why we can focus on \eqref{eq:tzij} instead when $j - i \asymp \log^3 n$. Note that \eqref{eq:oneujui} and \eqref{eq:mnplusbound} continue to hold when $j - i \asymp (\log n)^3$. Hence,
\beq
\tilde M_n^+(l_n^-, l_n^+) = \{1 + O_P(1/\sqrt{n})\}\tilde Z_n^+(l_n^-, l_n^+) + O_P(1/\sqrt{n}),
\eeq
which implies
\begin{align*}
\P\{\tilde Z_n^+(l_n^-, l_n^+) \le u_n(\tau - \eps)\} \le \P\{\tilde M_n^+(l_n^-, l_n^+) \le u_n(\tau)\} \le\P\{\tilde Z_n^+(l_n^-, l_n^+) \le u_n(\tau + \eps)\},
\end{align*}
for any $\eps > 0$. If we had established \eqref{eq:tznplus}, taking $\eps \to 0$ would yield \eqref{eq:tmnplusrestricted}. Now we turn to the mainstream of the proof.

\medskip
\noindent
{\sc Proof of \eqref{eq:local2}}. We will prove this following a similar strategy as in \citet{kabluchko2014limiting}. Necessary adjustments are still needed since \citet{kabluchko2014limiting} focused on $Z_{i, j}^+$ while we are dealing with $\tilde Z_{i, j}^+$. We will present the parts that need to be adjusted and refer to their results when nothing needs to be changed. 

First we work on $\cQ_n$. For any $\tau \in \mathbb R$ and $a \ge 0$, let $l_n = a(\log n)^3$ and define
\beq\label{eq:onedef2}
P_n(s) = \P\bigg(\frac{S_{l_n}^+}{\sqrt{l_n - S_{l_n}^+}} \ge u_n -\frac{s}{u_n}  \bigg).
\eeq
Define
\beq
b_n := \frac{u_n - s/u_n}{\sqrt{l_n}},
\eeq
 for ease of notation. Since $u_n^3 \propto \sqrt{l_n}$ and $b_n \sim \sqrt{2/a}/\log n\to 0$, for fixed $s > 0$ with sufficiently large $n$, with the transformation \eqref{eq:tzijtransform}, Lemma \ref{lem:moderatedev} and Taylor's expansion
\begin{align}
P_n(s) &= \P\bigg\{  \frac{S_{l_n}^+}{\sqrt{l_n}} \ge \sqrt{l_n}g^+(b_n)\bigg\}\nonumber\\ 
&\sim \frac{1}{\sqrt{2\pi}u_n}\exp\bigg\{  -\frac{(u_n - s/u_n)^2}{2}\frac{2I^+(g^+(b_n ))}{b_n^2}  \bigg\}\nonumber\\
&= \frac{1}{\sqrt{2\pi}u_n}\exp\bigg\{-\frac{(u_n - s/u_n)^2}{2}\bigg(1 - \frac{1}{3}b_n\bigg) + o(1)\bigg\}\nonumber\\
&\sim \frac{1}{2\sqrt{\pi}}e^{s + \frac{\sqrt{2}}{3} a^{-1/2}} \frac{e^{-\tau}\log n}{n} \label{eq:pns2}.
\end{align}
Recall that $\bbT_r(x, y)$ is defined in \eqref{T}. Define $q_n = (\log n)^2$. By the same techniques in the proof of Lemma \ref{lem:weakcontroltzijplus} we have
\begin{align*}
\cQ_n & = \P\bigg(\max_{(i, j) \in \bbT_{Bq_n}(x, x + l_n)} \tilde Z_{i, j}^+ \ge u_n\bigg) \\
&= \P\bigg[\max_{(i, j) \in \bbT_{Bq_n}(x, x + l_n)} \bigg\{S_j^+ - S_i^+ - (j - i)g^+\bigg(\frac{u_n}{j - i}\bigg)\bigg\} \ge 0\bigg]\\
&= \P\bigg[\max_{0 \le k_1, k_2 \le Bq_n} \bigg\{S_{k_1}^{(1)+} + S_{k_2}^{(2)+} -(l_n + k_1 + k_2)g^+\bigg(\frac{u_n}{l_n + k_1 + k_2}\bigg)\bigg\} + S_{l_n}^+ \ge 0\bigg]\\
&= P_n(0)\bigg\{  1 + \int_0^\infty G_n(s)d\nu_n(s)  \bigg\},
\end{align*}
where $P_n(s)$ defined in \eqref{eq:onedef2} is actually the probability distribution of $V_{l_n, u_n}$, defined in \eqref{eq:vludefn}. Therein
\begin{align*}
G_n(s) :=& \P\bigg[\max_{0 \le k_1, k_2 \le Bq_n} \bigg\{ S_{k_1}^{(1)+} + S_{k_2}^{(2)+} - (l_n + k_1 + k_2)g^+\bigg(\frac{u_n}{\sqrt{l_n + k_1 + k_2}}\bigg)\bigg\} \\
&~~~~~+ l_n \cdot g^+\bigg(\frac{u_n - s/u_n}{\sqrt{l_n}}\bigg) \ge 0\bigg], 
\end{align*}
and
\beq
\nu_n(\cdot) := P_n(\cdot)/P_n(0).
\eeq 
It is immediate that the first and second conditions in Lemma~\ref{lem:interchange} hold by directly mimicking the details in the proof of Lemma 4.3 in \citep{kabluchko2014limiting}, that is, for any fixed $s > 0$ and any sequence $s_n \to s$,
\beq
\lim_{n \to \infty}G_n(s_n) = \P( M_1 + M_2 \ge s ),
\eeq
and
\beq
\lim_{n \to \infty}\nu_n([0, s)) = \lim_{n \to \infty} \frac{P_n(s)}{P_n(0)} = e^s.
\eeq
$M_1$ and $M_2$ are independent copies with the same distribution as \beq
M  = \sup_{t \in [0, a^{-1}B]}  \{ \sqrt{2}W(t) - t \}, 
\eeq
where $W(t)$ is a standard Brownian motion (similar but more detailed arguments can be found in the proof of lemma 4.3 in \cite{kabluchko2011extremes}). To verify the third condition in Lemma~\ref{lem:interchange}, we need to bound the integral $ \int_0^\infty G_n(s)d\nu_n(s)$ from above. This can be immediately completed by using Lemma \ref{lem:weakcontroltzijplus}. Hence applying Lemma \ref{lem:interchange} completes the proof of \eqref{eq:local2}, where 
\beq
H(x) := \E\{\sup_{t \in [0, x]} e^{\sqrt{2}W(t) - t}\}\text{, }x > 0,
\eeq
therein. 

\medskip
\noindent
{\sc Proof of \eqref{eq:largelocalrate}}. Define $w_n = (\log n)^3$. For $z \in \bbZ$, define 
\beq
\bbJ_n(z) = \{(i, j) \in \bbI: z \le i < z + w_n, j - i \in [l_n^-, l_n^+]\}.
\eeq
To derive the rate of $\P(\max_{(i, j) \in \bbJ_n(z)} \tilde Z_{i, j}^+ \ge u_n )$, by translation invariance we may take $z = 0$.
Let $\delta_n$ be a real sequence satisfying  $\delta_n =o (w_n)$ and $q_n = o(\delta_n)$, e.g. $\delta_n = (\log n)^{2.5}$. For $B \in \bbN$, we introduce the following two-dimensional discrete grids with mesh size $Bq_n$:
\beq
\cJ_n(B) = \{(x, y) \in  Bq_n\bbZ \times Bq_n\bbZ : x \in [-\delta_n, w_n + \delta_n], y - x \in [l_n^- - \delta_n, l_n^+ + \delta_n]\},
\eeq
\beq
\cJ_n'(B) = \{(x, y) \in  Bq_n\bbZ \times Bq_n\bbZ : x \in [\delta_n, w_n -\delta_n], y - x \in [l_n^- + \delta_n, l_n^+ - \delta_n]\}.
\eeq
By Bonferroni inequality,
\beq
S_n'(B) - S_n''(B) \le \P\bigg(  \max_{(i, j) \in \bbJ_n(0)} \tilde Z_{i, j}^+ \ge u_n  \bigg) \le S_n(B),
\eeq
where
\beq\label{eq:firstpart}
S_n(B) = \sum_{(x, y) \in \cJ_n(B)}\P\bigg(  \max_{(i, j) \in \bbT_{Bq_n}(x, y)} \tilde Z_{i, j}^+ \ge u_n  \bigg),
\eeq
\beq
S_n'(B) = \sum_{(x, y) \in \cJ_n'(B)}\P\bigg(  \max_{(i, j) \in \bbT_{Bq_n}(x, y)} \tilde Z_{i, j}^+ \ge u_n  \bigg),
\eeq
and
\beq\label{eq:snppb}
S_n''(B) = \sum_{(x_1, y_1), (x_2, y_2) }\P\bigg(  \max_{(i, j) \in \bbT_{Bq_n}(x_1, y_1)} \tilde Z_{i, j}^+ \ge u_n, \max_{(i, j) \in \bbT_{Bq_n}(x_2, y_2)} \tilde Z_{i, j}^+ \ge u_n  \bigg),
\eeq
where the summation is taken over $(x_1, y_1) \neq  (x_2, y_2) \in  \cJ_n'(B)$. As long as we can show
\beq\label{eq:snbupper}
\lim_{B \to \infty}\limsup_{n \to \infty}nw_n^{-1}S_n(B) \le e^{-\tau}\int_{A_1}^{A_2} \Lambda_1(a)da,
\eeq
\beq\label{eq:snpblower}
\lim_{B \to \infty}\liminf_{n \to \infty}nw_n^{-1}S_n'(B) \ge e^{-\tau}\int_{A_1}^{A_2} \Lambda_1(a)da, 
\eeq
and
\beq\label{eq:snppbzero}
\lim_{B \to \infty}\limsup_{n \to \infty}nw_n^{-1}S_n''(B) = 0,
\eeq
\eqref{eq:largelocalrate} will follow immediately. 
The proof of \eqref{eq:snpblower} is almost identical to that of \eqref{eq:snbupper}, so we only focus on proving \eqref{eq:snbupper} based on the dominated convergence theorem. Define
\beq
\cL_n(B) = Bq_n \bbZ\cap [l_n^- -\delta_n, l_n^+ +\delta_n],
\eeq
such that $|\cL_n(B)| \sim (A_2 - A_1) (\log n)/B$. Since the probability on the right-hand side of \eqref{eq:firstpart} depends only on $l := y - x$, by translation invariance we have
\beq
S_n(B) \le \frac{w_n + \delta_n}{Bq_n} \sum_{l \in \cL_n(B)}\P\bigg(  \max_{(i, j) \in T_{Bq_n}(0, l)} \tilde Z_{i, j}^+ \ge u_n  \bigg).
\eeq
Next we apply \eqref{eq:local2} to bound each probability with $l$ fixed and replace $ (B q_n)^{-1} \sum_{l \in \cL_n(B)}$ by an integral as $n\to \infty$. By \eqref{eq:local2} and \eqref{eq:pns2},
\beq
\lambda_{n,B}(a) : = \frac{n}{\log n}\P\bigg(  \max_{(i,j) \in T_{Bq_n}(0, l_{n,B}(a))} \tilde Z_{i,j}^+ \ge u_n  \bigg) \to \frac{1}{2\sqrt{\pi}}e^{ \frac{\sqrt{2}}{3} a^{-1/2} - \tau}\bigg\{  1 + H^2\bigg(\frac{B}{a}\bigg)  \bigg\},
\eeq
as $n \to \infty$, where
\beq
l_{n,B}(a) = \max\{l \in Bq_n\bbZ: l \le a w_n \}. 
\eeq
The function $\lambda_{n,B}(a)$ takes constant values on sub-intervals with widths $Bq_n/w_n = B/\log n$. It follows that
\beq
S_n(B) \le \frac{w_n + \delta_n}{B^2n} \sum_{l \in \cL_n(B)} \frac{B\lambda_{n,B}(a)}{\log n} = \frac{w_n + \delta_n}{B^2 n}\int_{A_1 - \frac{2\delta_n}{w_n}}^{A_2 + \frac{2\delta_n}{w_n}}\lambda_{n,B}(a)da.
\eeq
From Lemma \ref{lem:weakcontroltzijplus}, 
we can upper bound the integrand $\lambda_{n,B}(a)$ by an integrable function that is independent of $n$. Therefore, applying Fatou's lemma on $\limsup$ gives
\beq\label{eq:domchange}
\limsup_{n \to \infty}nw_n^{-1}S_n(B) \le e^{-\tau}\int_{A_1}^{A_2} \frac{a^2\Lambda_1(a)}{B^2}\bigg\{ 1 + H^2\bigg(\frac{B}{a}\bigg)\bigg\}da.
\eeq
This result holds for any $B \in \bbN$. Note that $\lim_{B \to \infty} H(B)/B = 1$. Letting $B\to \infty$, we arrive at \eqref{eq:snbupper}. 

To prove \eqref{eq:snppbzero}, we bound $S_n''(B)$ by similar quantities of $Z_{i, j}^+$, which allows us to use results in \citet{kabluchko2014limiting} immediately. For any interval $(x, y)$ define the event
\beq
E_n(x, y) = \bigg\{\max_{(i,j) \in \bbT_{Bq_n}(x, y)} \tilde Z_{i,j}^+ \ge u_n\bigg\}.
\eeq
Note that
\beq
\frac{g^+(x)}{x} = \frac{1}{2}(\sqrt{x^2 + 4} - x)  \ge 1 - \frac{x}{2}, \text{ when }x \to 0.
\eeq
When $y - x \propto (\log n)^3$, $u_n/(y - x) \propto 1/(\log n)$,
\begin{align*}
E_n(x, y) &= \bigg\{\max_{0 \le l_1, l_2 \le Bq_n} \bigg\{S_{y + l_2}^+ - S_{x - l_1}^+ - (y - x + l_1 + l_2)g^+\bigg(\frac{u_n}{\sqrt{y - x + l_1 + l_2}}\bigg)\bigg\} \ge 0  \bigg\}\\
&\subset \bigg\{\max_{0 \le l_1, l_2 \le Bq_n} \frac{S_{y + l_2}^+ - S_{x - l_1}^+}{\sqrt{y - x + l_1 + l_2}} \ge \sqrt{y - x + l_1 + l_2} g^+\bigg(\frac{u_n}{\sqrt{y - x + l_1 + l_2}}\bigg)  \bigg\}\\
&\subset \bigg\{\max_{(i,j) \in \bbT_{Bq_n}(x, y)} Z_{i,j}^+  \ge u_n(\tau)\bigg(1 - \frac{u_n}{2\sqrt{y - x + l_1 + l_2}}\bigg)\bigg\}\\
&\subset \bigg\{\max_{(i,j) \in \bbT_{Bq_n}(x, y)} Z_{i,j}^+  \ge u_n(\tau - 0.1)\bigg\}.
\end{align*}
Therefore,
\begin{align*}
&\P\{ E_n(i_1, j_1) \cap E_n(i_2, j_2) \} \\
&\le \P\bigg[\bigg\{\max_{(i,j) \in \bbT_{Bq_n}(i_1, j_1)} Z_{i,j}^+  \ge u_n(\tau - 0.1)\bigg\} \bigcap \bigg\{\max_{(i,j) \in \bbT_{Bq_n}(i_2, j_2)} Z_{i,j}^+  \ge u_n(\tau - 0.1)\bigg\}\bigg].
\end{align*}
This allows us to work on $Z_{i, j}^+$ instead. Directly applying Lemma 4.12, Lemma 4.14, Lemma 4.15 and Lemma 4.16 in \citep{kabluchko2014limiting} yields \eqref{eq:snppbzero}.

\medskip
\noindent
{\sc Proof of \eqref{eq:tznplus}}.  We will temporarily adopt the notations in \citet{arratia1989two}. Define 
\beq
I = \{\alpha \in \bbN: \alpha w_n \le n\},
\eeq
which implies $|I| \le n/w_n$. For any $\alpha \in I$, define
\beq
X_\alpha = \mathbbm{1}\{\max_{(i, j) \in \bbJ_n(\alpha w_n)} \tilde Z_{i, j}^+ \ge u_n \}, 
\eeq
\beq
p_\alpha = \P(X_\alpha),
\eeq
and
\beq
B_\alpha = \{\beta \in I: |(\beta - \alpha)w_n| \le l_n^+ + w_n\}.
\eeq
Hence $|B_\alpha| \le A_2 + 1$. To apply Theorem 1 in \citep{arratia1989two}, we need to show that 
\beq
b_1 := \sum_{\alpha \in I}\sum_{\beta \in B_\alpha} p_\alpha p_\beta, 
\eeq
\beq
b_2 := \sum_{\alpha \in I}\sum_{\alpha \neq \beta \in B_\alpha} p_{\alpha\beta}, \text{ where } p_{\alpha\beta} := \E(X_\alpha X_\beta),
\eeq
and 
\beq
b_3' := \sum_{\alpha \in I} s_\alpha'
\eeq
therein vanish as $n \to \infty$, where
\beq
s_\alpha' := \E\bigg|\E\bigg(X_\alpha - p_\alpha \Big|\sum_{\beta \in I - B_\alpha} X_\beta \bigg)\bigg|
\eeq 
By the definition of $B_\alpha$, $X_\alpha - p_\alpha$ and $\sum_{\beta \in I - B_\alpha} X_\beta$ are independent. Hence $s_\alpha'  = 0$, so is $b_3'$. It follows from \eqref{eq:largelocalrate} that
\beq
b_1 \sim |I||B_\alpha| p_\alpha p_\beta \to 0. 
\eeq
With slight modification on \eqref{eq:largelocalrate},
\beq
\P\bigg(\max_{(i, j) \in \bbJ_n(\alpha w_n) \cup \bbJ_n(\beta w_n)} \tilde Z_{i, j}^+ \ge u_n \bigg) \sim e^{-\tau}\frac{2w_n}{n}\int_{A_1}^{A_2} \Lambda_1(a)da.
\eeq
This, together with \eqref{eq:largelocalrate}, implies
\beq
p_{\alpha\beta} = \P\bigg(\max_{(i, j) \in \bbJ_n(\alpha w_n)} \tilde Z_{i, j}^+ \ge u_n, \max_{(i, j) \in \bbJ_n(\beta w_n)} \tilde Z_{i, j}^+ \ge u_n \bigg) = o\bigg(\frac{w_n}{n}\bigg).
\eeq
Thus, 
\beq
b_2 \le |I||B_\alpha|\max_{\alpha \neq \beta} p_{\alpha\beta} \to 0. 
\eeq
Now, by Theorem 1 in \citep{arratia1989two},
\beq
\lim_{n\to \infty}\P\{\tilde Z_n^+(l_n^-, l_n^+) \le u_n\} = \lim_{n\to \infty}\P\bigg(\sum_{\alpha \in I}X_\alpha = 0\bigg) = e^{-\lambda},
\eeq
where
\beq
\lambda = \sum_{\alpha \in I}p_\alpha \to e^{-\tau}\int_{A_1}^{A_2} \Lambda_1(a)da. 
\eeq
Therefore,
\beq
\lim_{n\to \infty}\P\{\tilde M_n^+(l_n^-, l_n^+) \le u_n\} = \exp\bigg(-e^{-\tau}\int_{A_1}^{A_2} \Lambda_1(a)da\bigg),
\eeq
by the statement in the beginning of our proof.

\medskip
\noindent
{\sc Proof of \eqref{eq:tznplusbeyond}}. Divide $(l_n^+, n]$ into $(l_n^+, (\log n)^4]$, $((\log n)^4, n - (\log n)^4]$ and $(n - (\log n)^4, n]$. Within the first region, for any $k \in \bbN$, any pair $(i, j)$ with length $2^k (\log n)^3 \le j - i \le 2^{k + 1} (\log n)^3$ can be covered by the union of at most $2^{-k}n/\log n$ disjoint discrete squares of the form $\bbT_{2^k(\log n)^2}(x, x + j - i)$. By \eqref{eq:orderdiffup2}, 
\beq
1 - (\uj - \ui) \ge 1 -  1.1(\log n)^4/n,
\eeq
with probability tending to one. With these facts, by the union bound and Lemma \ref{lem:weakcontroltzijplus}, 
\begin{align*}
&\P\{\tilde M_n^+(l_n^+, (\log n)^4) \ge u_n\} \\
&\le \P\bigg\{  \max_{k: \log_2 A_2 \le k \le \log_2 (\log n)}\tilde M_n^+(2^k(\log n)^3, 2^{k + 1}(\log n)^3) \ge u_n \bigg\} \\
&\le \P\bigg\{  \max_{k:  \log_2 A_2 \le k \le \log_2 (\log n)}\tilde Z_n^+(2^k(\log n)^3, 2^{k + 1}(\log n)^3) \ge u_n(\tau - 0.1) \bigg\} \\
&\le \sum_{k \ge \log_2 A_2}2^{-k}\frac{n}{\log n}\P\bigg\{ \max_{(i, j) \in T_{2^k(\log n)^2}(0, 2^{k + 1}(\log n)^3)} \tilde Z_{i, j}^+ \ge u_n(\tau - 0.1)\bigg\} + \P(\Omega_n^c) \\
&\le C\sum_{k \ge \log_2 A_2}2^{-k}  + \P(\Omega_n^c).
\end{align*}
Taking $\limsup_{n \to \infty}$ and letting $A_2 \to \infty$ gives the desired result.

In the meantime, on $((\log n)^4, n - (\log n)^4]$, a finer examination of \eqref{eq:orderdiffup} and \eqref{eq:orderdifflow} yields
\beq
\bigg|\frac{n(\uj - \ui)}{j - i} - 1\bigg| = O_p\bigg(\frac{1}{\log n}\bigg).
\eeq
\eqref{eq:orderdiffup2} and \eqref{eq:orderdifflow2} imply
\beq
\bigg|\frac{1 - (\uj - \ui)}{1 - (j - i)/n} - 1\bigg| = O_p\bigg(\frac{1}{\log n}\bigg).
\eeq
Therefore,
\begin{align*}
\P\{\tilde M_n^+((\log n)^4, n - (\log n)^4) \ge u_n\} \le \P\{M_n^+(l_n^+, (\log n)^4) \ge u_n(\tau - 0.1)\} \to 0,
\end{align*}
by Theorem \ref{thm:mnplus}.

The proof of the region $(n - (\log n)^4, n]$ is immediate by following the proof for \eqref{eq:tmnpluslast}, which we omit here.
  \qed

\subsection{Proof of Theorem \ref{thm:tmnminus}}
Define
\beq
\tilde Z_{i, j}^- := \frac{S_j^- - S_i^-}{\sqrt{j - i + S_j^- - S_i^-}},
\eeq
and
\beq
g^-(a) := \frac{1}{2}(a\sqrt{a^2 + 4} + a^2).
\eeq
\beq\label{eq:iminusgminusbound}
I^-(g^-(s)) \ge s^2/2.
\eeq 
The theorem follows immediately after showing that
\beq
\limsup_{n \to \infty}\P(\tilde M_n^- \ge \eps\sqrt{n}) = 0,
\eeq
for any $\eps > 0$. This can be proved similarly by dividing the regions, transforming the statistic $\tilde M_{i, j}^-$ into $\tilde Z_{i, j}^-$,
combined with \eqref{eq:iminusgminusbound}. We omit the detail here.

\section*{Acknowledgements}
Andrew Ying was partially supported by the Achievement Rewards for College Scientists (ARCS) Scholarship. The authors strongly thanks for Professor Ery Arias-Castro for building up the introduction and providing the motivation. The authors would also like to thank for Professor Qi-Man Shao, Professor Xiao Fang, Professor Hock Peng Chan, and Professor David O.~Siegmund for stimulating discussions and pointers to the literature.

\bibliographystyle{chicago}
\bibliography{ref}
\end{document}